\documentclass[10pt]{article}

\input diagram.tex
\usepackage{amsmath} 
\usepackage{mathrsfs}
\usepackage{amssymb} 
\usepackage{theorem}
\usepackage{rotating}
\usepackage{stmaryrd}
\usepackage[all,cmtip]{xy}
\usepackage{pict2e}
\theoremstyle{change} 
\newtheorem{theorem}{Theorem}[section] 
\newtheorem{lemma}[theorem]{Lemma} 
\newtheorem{proposition}[theorem]{Proposition}
\newtheorem{corollary}[theorem]{Corollary}
\theorembodyfont{\rmfamily} 
\newtheorem{remark}[theorem]{Remark}
\newtheorem{example}[theorem]{Example}

\newtheorem{definition}[theorem]{Definition}
\newtheorem{notation}[theorem]{Notation}
\newtheorem{nothing}[theorem]{} 

\newenvironment{proof}{\noindent{\bf Proof}\ }{\qed\bigskip}

\renewcommand{\le}{\leqslant}
\renewcommand{\ge}{\geqslant} 
\renewcommand{\leq}{\leqslant}
\renewcommand{\geq}{\geqslant}
\renewcommand{\unlhd}{\trianglelefteqslant}

\renewcommand{\marginpar}[1]{}

\newcommand{\Aut}{\mathrm{Aut}}
\newcommand{\Atilde}{\tilde{A}}

\newcommand{\BIGOP}[1]
  {\mathop{\mathchoice
  {\raise-0.22em\hbox{\huge $#1$}}
  {\raise-0.05em\hbox{\Large $#1$}}{\hbox{\large $#1$}}{#1}}}
\newcommand{\bigtimes}{\BIGOP{\times}}

\newcommand{\calC}{\mathcal{C}}

\newcommand{\calS}{\mathcal{S}}
\newcommand{\calT}{\mathcal{T}}
\newcommand{\catC}{\catfont{C}}

\newcommand{\catCtilde}{\tilde{\catC}}
\newcommand{\catB}{\catfont{B}}
\newcommand{\catD}{\catfont{D}}
\newcommand{\catfont}{\mathsf}

\newcommand{\etilde}{\tilde{e}}
\newcommand{\End}{\mathrm{End}}

\newcommand{\Hom}{\mathrm{Hom}}

\newcommand{\GLH}{\llap{\phantom{L}}_G L_H}
\newcommand{\GLMK}{\llap{\phantom{(L*M)}}_G (L*M)_K}
\newcommand{\gU}{\lexp{g}{U}}
\newcommand{\Hd}{\mathrm{Hd}}
\newcommand{\HMK}{\llap{\phantom{M}}_H M_K}
\newcommand{\HMG}{\llap{\phantom{M}}_HM_G}
\newcommand{\id}{\mathrm{id}}
\newcommand{\Inf}{\mathrm{Inf}}

\newcommand{\Inn}{\mathrm{Inn}}

\newcommand{\Mor}{\mathrm{Mor}}
\newcommand{\myiso}{\buildrel\sim\over\to}

\newcommand{\lexp}[2]{\setbox0=\hbox{$#2$} \setbox1=\vbox to
                 \ht0{}\,\box1^{#1}\!#2}

\newcommand{\link}{-}
\newcommand{\Ltilde}{\tilde{L}}
\newcommand{\liso}{\buildrel\sim\over\longrightarrow}

\newcommand{\Mat}{\mathrm{Mat}}
\newcommand{\Mtilde}{\tilde{M}}

\newcommand{\NN}{\mathbb{N}}

\newcommand{\Ob}{\mathrm{Ob}}
\newcommand{\Out}{\mathrm{Out}}
\newcommand{\PP}{\mathbb{P}}
\newcommand{\qed}{\nobreak\hfill
                   \vbox{\hrule\hbox{\vrule\hbox to 5pt
                   {\vbox to 8pt{\vfil}\hfil}\vrule}\hrule}}
\newcommand{\QQ}{\mathbb{Q}}
\newcommand{\Rad}{\mathrm{Rad}}

\newcommand{\scrA}{\mathscr{A}}
\newcommand{\scrD}{\mathscr{D}}
\newcommand{\scrP}{\mathscr{P}}
\newcommand{\scrS}{\mathscr{S}}
\newcommand{\scrX}{\mathscr{X}}
\newcommand{\scrY}{\mathscr{Y}}
\newcommand{\shat}{\hat{s}}
\newcommand{\stab}{\mathrm{stab}}
\newcommand{\scrStilde}{\tilde{\scrS}}
\newcommand{\startilde}{\tilde{*}}

\newcommand{\that}{\hat{t}}

\newcommand{\tr}{\mathrm{tr}}

\newcommand{\Ttilde}{\tilde{T}}

\newcommand{\Wtilde}{\tilde{W}}

\newcommand{\ZZ}{\mathbb{Z}}

\title{A Ghost Algebra of the Double Burnside Algebra in Characteristic Zero\footnote{{\bf MR Subject Classification:} 19A22, 20C20. {\bf Keywords:} Burnside ring, double Burnside ring, mark homomorphism, ghost ring, idempotent condensation, biset, biset functor, twisted category algebra}}
\author{\small Robert Boltje\\
  \small Department of Mathematics\\
  \small University of California\\
  \small Santa Cruz, CA 95064\\
  \small U.S.A.\\
  \small boltje@ucsc.edu
  \and
  \small Susanne Danz\\
  \small Department of Mathematics\\ 
  \small University of Kaiserslautern\\
  \small P.O. Box 3049\\
  \small 65653 Kaiserslautern\\
  \small Germany\\
  \small danz@mathematik.uni-kl.de}
\date{March 6, 2012\\ {\small revised: June 19, 2012}}

\begin{document}

\sloppy

\maketitle


\begin{abstract}
For a finite group $G$, we introduce a multiplication on the $\QQ$-vector space with basis $\scrS_{G\times G}$, the set of subgroups of $G\times G$. The resulting $\QQ$-algebra $\Atilde$ can be considered as a ghost algebra for the double Burnside ring $B(G,G)$ in the sense that the mark homomorphism from $B(G,G)$ to $\Atilde$ is a ring homomorphism. Our approach interprets $\QQ B(G,G)$ as an algebra $eAe$, where $A$ is a twisted monoid algebra and $e$ is an idempotent in $A$. The monoid underlying the algebra $A$ is again equal to $\scrS_{G\times G}$ with multiplication given by composition of relations (when a subgroup of $G\times G$ is interpreted as a relation between $G$ and $G$). The algebras $A$ and $\Atilde$ are isomorphic via M\"obius inversion in the poset $\scrS_{G\times G}$. As an application we improve results by Bouc on the parametrization of simple modules of $\QQ B(G,G)$ and also of simple biset functors, by using results of Linckelmann and Stolorz on the parametrization of simple modules of finite category algebras. Finally, in the case where $G$ is a cyclic group of order $n$, we give an explicit isomorphism between $\QQ B(G,G)$ and a direct product of matrix rings over group algebras of the automorphism groups of cyclic groups of order $k$, where $k$ divides $n$.
\end{abstract}


\section{Introduction}

The main goal of this paper is to find a ghost ring for the double Burnside 
ring $B(G,G)$ of a finite group $G$, as an analogue of the classical ghost ring
of the usual 
Burnside ring $B(G)$, and as a generalization of the results in \cite{BD11}, where 
this was achieved for the subring $B^\lhd(G,G)$ of $B(G,G)$, which is additively 
generated by the classes of left-free $(G,G)$-bisets. This goal is achieved over the field $\QQ$ of rational numbers.

\smallskip
Recall that the {\it Burnside ring} $B(G)$, the Grothendieck ring of the category of 
finite left $G$-sets with respect to disjoint unions and direct products, is 
embedded via the {\it mark homomorphism} $\rho_G :B(G)\to \prod_{U\le G}\ZZ$, 
$[X]\mapsto (|X^U|)_{U\le G}$, into a direct product of copies of $\ZZ$, also called the 
{\it ghost ring} of $B(G)$, a ring with a considerably simpler structure than $B(G)$ itself. Here, $X$ denotes a left $G$-set,
$[X]$ denotes its image in $B(G)$, and $X^U$ denotes the set of $U$-fixed points of $X$, for a subgroup $U$ of $G$. Every element in the image of $\rho_G$ is constant on conjugacy classes of subgroups of $G$. After tensoring with $\QQ$ one obtains a $\QQ$-algebra isomorphism 
\begin{equation*}
  \rho_G\colon \QQ B(G)\liso (\prod_{U\le G}\QQ)^G \subseteq \prod_{U\le G} \QQ
\end{equation*}
onto the subalgebra of $G$-fixed points under the conjugation action. This isomorphism has become one of the main tools to obtain answers to questions related to the ring structure of $B(G)$: for instance, Dress determined the primitive idempotents of $B(G)$, cf.~\cite{D}, and Yoshida characterized the units of $B(G)$, cf.~\cite{Yoshida}, using the mark homomorphism. 

\smallskip
The double Burnside ring $B(G,G)$ is the Grothendieck ring of the category of finite $(G,G)$-bisets with respect to disjoint unions and the tensor product of bisets over $G$. In contrast to $B(G)$, it is not commutative when $G\neq 1$. 
The Burnside ring and the double Burnside ring are linked through a natural embedding $B(G)\to B(G,G)$. The double Burnside ring and related constructions have led to the theory of biset functors, initiated by Bouc in \cite{BoucJA}, see also the more recent book \cite{BoucSLN}. This theory was the main tool in the determination of the Dade group, an invariant of a $p$-group that is important in the modular representation theory of finite groups. It was achieved in a sequence of papers by various authors, including 
Bouc, Th\'evenaz and Carlson, cf.~\cite{BoucThevenaz00} and \cite{Thevenaz}. Another application of biset functors was the complete determination of the unit group of $B(P)$ for a $p$-group $P$, cf.~\cite{Bouc2007}. Bisets also led to surprising and interesting invariants, called stabilizing bisets, of representations of finite groups over arbitrary commutative base rings, cf.~\cite{BoucThevenaz2010}. 
Moreover, there is a connection of the double Burnside ring with algebraic topology, cf.~\cite{MartinoPriddy}, and fusion systems, cf.~\cite{BLO1}, ~\cite{RS}, although these connections only use the subring $B^\lhd(G,G)$ of $B(G,G)$.

\smallskip
Given the renewed interest in the double Burnside ring, the potential for applications, and the usefulness of the ghost ring of $B(G)$, it seems desirable to determine a ghost algebra of $B(G,G)$ in the following sense. The double Burnside ring can, as an abelian group, be canonically identified with the Burnside group $B(G\times G)$, and this way one has the usual additive embedding via the mark homomorphism $\rho_{G,G}\colon B(G,G)\to \ZZ\scrS_{G\times G}$ into the free $\ZZ$-module over the set $\scrS_{G\times G}$ of subgroups of $G\times G$, so that one obtains an isomorphism of $\QQ$-vector spaces,
\begin{equation*}
  \rho_{G,G}\colon \QQ B(G,G)\liso (\QQ\scrS_{G\times G})^{G\times G}\subseteq \QQ\scrS_{G\times G}\,,
\end{equation*}
onto the subspace of $G\times G$-fixed points under the conjugation action of $G\times G$ on the set of its subgroups, $\scrS_{G\times G}$. The multiplication on $\QQ B(G,G)$ induces a unique algebra structure on the fixed point set that turns $\rho_{G,G}$ into a $\QQ$-algebra isomorphism. It is natural to ask what this multiplication on $(\QQ\scrS_{G\times G})^{G\times G}$ looks like. One can also ask whether there is a \lq natural\rq\ $\QQ$-algebra structure on the bigger vector space $\QQ\scrS_{G\times G}$ extending the multiplication on the fixed points. Of course, such an extension might not be unique. Using the (single) Burnside ring $B(G)$ and its mark homomorphism as a model, one would also hope that such a $\QQ$-algebra structure on $\QQ\scrS_{G\times G}$ should be simpler than the one of $\QQ B(G,G)$ itself. Also, it is reasonable to expect that this multiplication should involve the construction of \lq composition\rq, $L*M$, of two subgroups $L$ and $M$ of $G\times G$, given by 
\begin{equation*}
  L*M:=\{(x,y)\in G\times G\mid \exists g\in G\colon (x,g)\in L, (g,y)\in M\}\,,
\end{equation*}
which is used in the explicit description of the multiplication in $\QQ B(G,G)$ by a Theorem of Bouc, cf.~Proposition~\ref{prop Mackey}.

\smallskip
The key observation leading to answers to the above questions is that the additive map $\alpha_{G,G}\colon \QQ B(G,G)\to \QQ \scrS_{G\times G}$, $[X]\mapsto \sum_{x\in X}\stab_{G\times G}(x)$, is almost multiplicative when $\QQ\scrS_{G\times G}$ is viewed as monoid algebra over the monoid $(\scrS_{G\times G},*)$. This map is multiplicative if one introduces a $2$-cocycle on the monoid $\scrS_{G\times G}$ and uses the twisted monoid algebra structure on $\QQ\scrS_{G\times G}$. For the purpose of this introduction, we denote this algebra by $A$. The map $\alpha_{G,G}$ has the effect of \lq straightening out\rq\ the Mackey-type formula for the multiplication in $B(G,G)$, by omitting the summation over double cosets. Moreover, $\alpha_{G, G}$ is injective with image $eAe$ for a certain idempotent $e$ of $A$. Thus, the double Burnside algebra and its representation theory can be viewed as the result of idempotent condensation applied to a twisted monoid algebra. The $\QQ$-linear isomorphism $\zeta_{G,G}\colon\QQ\scrS_{G\times G}\liso \QQ\scrS_{G\times G}$, $L\mapsto \sum_{L'\le L}L'$, transports the twisted monoid algebra structure $A$ on $\QQ\scrS_{G\times G}$ to another algebra structure, which we denote by $\Atilde$ on the same $\QQ$-vector space. It turns out that the composition $\zeta_{G,G}\circ\alpha_{G,G}$ is equal to the mark homomorphism $\rho_{G,G}$ and we obtain a commutative diagram (omitting the indices of the maps)
\begin{diagram}[100]
  \QQ B(G,G) & \movearrow(0,0){\Ear[20]{\alpha}} & \movevertex(30,0){(\QQ\scrS_{G\times G})^{G\times G}=eAe} & 
                   \movearrow(50,0){\subseteq} & \movevertex(50,-5){\QQ\scrS_{G\times G} = A} &&
   & \seaR{\rho} & \movearrow(10,0){\saR{\zeta}} &            & \movearrow(30,0){\saR{\zeta}} &&
   &               & \movevertex(30,0){(\QQ\scrS_{G\times G})^{G\times G}=\etilde\Atilde\etilde} & \movevertex(50,0){\subseteq}  & 
                \movevertex(50,-5){\QQ\scrS_{G\times G}=\Atilde} &&
\end{diagram}
of algebra isomorphisms, with $\etilde:=\zeta(e)$.

\smallskip
The first part of the paper is arranged as follows. Section~\ref{sec basics} recalls the basic notions related to bisets and the double Burnside ring, Section~\ref{sec embeddings} establishes the properties related to the above diagram, and Section~\ref{sec aLMN} explicitly determines the multiplication formula for the basis elements $L$ and $M$ in $\Atilde$, cf.~Theorem~\ref{thm aLMN}. It turns out that, in general, the product of $L$ and $M$ is an alternating sum of very few subgroups $N$ of $G\times G$; and if the image of the right projection of $L$ to $G$ and the image of the left projection of $M$ to $G$ are different then the product is equal to $0$. 

\smallskip
The rest of the paper is dedicated to applications of the approach to view 
$\QQ B(G,G)$ as a condensed algebra $eAe$ of the twisted monoid algebra $A$. The first application is the study of simple $\QQ B(G,G)$-modules (and also simple biset functors) by first determining the simple $A$-modules and then using Green's theory, cf.~\cite[Section~6.2]{Gr2}. Recall that, by Green's theory, every simple $eAe$-module is of the form $eS$, where $S$ is a simple $A$-module. This way, there is a natural injective map from the set of isomorphism classes of simple $eAe$-modules into the set of isomorphism classes of simple $A$-modules. The class of $S$ is contained in the image of this embedding if and only if $e\cdot S\neq \{0\}$, a condition that often cannot be checked easily. Section~\ref{sec category algebras} recalls Green's theory and also results on the parametrization of simple modules of twisted category algebras due to Linckelmann--Stolorz (\cite{LS}), and on the structure of inverse category algebras due to Linckelmann (\cite{L}). This is used in Section~\ref{sec simples of A} to obtain a parametrization of the simple $A$-modules (cf.~Theorem~\ref{thm simples of A}). In Section~\ref{sec simple biset functors} we give an explicit criterion when a simple $A$-module $S$ satisfies ~$e\cdot S\neq\{0\}$, cf.~Theorem~\ref{thm simples of RB(G,G)}. This leads to improvements, but only in characteristic $0$, of results by Bouc,  as given in \cite{BoucSLN}, on the parametrization of simple $\QQ B(G,G)$-modules, cf.~Theorem~\ref{thm simple RB(G,G)-modules}, and simple biset functors, cf.~Theorem~\ref{thm simple biset functors}. By Bouc's results, the simple $\QQ B(G,G)$-modules are parametrized by pairs $(H,W)$, where $H$ is a subquotient of $G$ and $W$ is a simple $\QQ\Out(H)$-module. But not all simple modules $W$ occur in this parametrization. We give a purely representation-theoretic necessary condition for $W$ to occur in this parametrization. Finally, in Section~\ref{sec cyclic}, we prove that the algebra $\Atilde$ is the category algebra of an inverse category when $G$ is cyclic. In this case we can use the results from Section~\ref{sec category algebras} to determine an explicit isomorphism between $\Atilde$ and a product of matrix rings over group algebras,  cf.~Theorem~\ref{thm A matrix dec}.

\smallskip
For simplicity we have, so far, only given an account of the results in the setting where one considers a single finite group $G$. Throughout the paper we develop a more general theory with respect to double Burnside groups $RB(G,H)$, where $R$ is an appropriate commutative ring and $B(G,H)$ is the double Burnside group of two finite groups $G$ and $H$. Only this more general setting leads to a description of the category of biset functors as a category of modules over a condensed algebra of a twisted category algebra, cf.~Example~\ref{ex Bouc algebra}(c). This is where one really needs category algebras and not only monoid algebras. Our set-up of the underlying category on which biset functors are defined is slightly more general than the one used in \cite{BoucSLN} by using Condition~(\ref{eqn closed}) in Section~\ref{sec simples of A}, which avoids to require that all group isomorphisms are present in the category. For results on simple biset functors see also \cite{BST}.

\medskip
Unless specified otherwise, all groups occurring in this paper will
be finite, and $R$ will always denote an associative commutative unitary ring.

\bigskip
{\bf Acknowledgements.}\, The authors' research on this project was supported through
a Marie Curie Intra-European Fellowship (grant PIEF-GA-2008-219543), 
a `R\"uckkehrstipendium' of the German Academic Exchange Service (grant D/11/05654), 
and a
`Research in Pairs' Grant from the Mathematisches Forschungsinstitut 
Oberwolfach (MFO) in February 2012.
It is a pleasure to thank the Universities of Oxford, Jena and Kaiserslautern 
as well as the 
MFO, where parts of this paper were written, for their kind hospitality.

\smallskip The authors would like to thank the referee for providing Lemma~\ref{lem referee}, including a proof, together with the suggestion to use it in the proof of Theorem~\ref{thm aLMN}. This yields a much more conceptual understanding of Theorem~\ref{thm aLMN} than its original ad-hoc proof.


\section{Preliminaries}\label{sec basics}

Throughout this section, let $G$, $H$, and $K$ be finite groups.
We begin by fixing some general notation, and summarize some
basic results that will be essential for this paper. More details can be found in \cite[Part I]{BoucSLN}.

\begin{nothing}\label{noth groups}{\bf Subgroups and sections.}\;
(a)\, For every subgroup $U$ of $G$ and every $g\in G$, we set
$U^g:=g^{-1}Ug$ and ${}^gU:=gUg^{-1}$. Moreover, we denote by $c_g$ the
automorphism $G\to G:\, h\mapsto ghg^{-1}$. 

The set of all subgroup of $G$ will be denoted by $\scrS_G$, and
$\tilde{\scrS}_G\subseteq \scrS_G$ will denote a transversal for the conjugacy classes
of subgroups of $G$. For $U\leq G$, we denote its $G$-conjugacy class by $[U]_G$ and its isomorphism class by $[U]$.

\smallskip
(b)\, By a {\it section} of $G$ we understand a pair $(U,V)$ such that
$V\unlhd U\leq G$. The group $G$ acts on the set of all its sections via conjugation:
if $(U,V)$ is a section of $G$ and if $g\in G$ then we set ${}^g(U,V):=({}^gU,{}^gV)$ and
$(U,V)^g:=(U^g,V^g)$.

\smallskip
(c)\, We denote the canonical projections
$G\times H\to G$ and $G\times H\to H$ by $p_1$ and $p_2$, respectively,
and for every $L\leq G\times H$ we further set
$$k_1(L):=\{g\in G\mid (g,1)\in L\}\quad \text{ and }\quad k_2(L):=\{h\in H\mid (1,h)\in L\}\,.$$
Note that, for $i=1,2$, we have $k_i(L)\unlhd p_i(L)$ and $p_i$ induces
a group isomorphism $\bar{p}_i:L/(k_1(L)\times k_2(L))\to p_i(L)/k_i(L)$. 
Thus one has a group isomorphism
$$\eta_L:=\bar{p}_1\circ \bar{p}_2^{-1}:p_2(L)/k_2(L)\myiso p_1(L)/k_1(L)$$
with $\eta_L(hk_2(L))=gk_1(L)$, for $(g,h)\in L$. 
In this way one obtains a bijection between the
set $\scrS_{G\times H}$ and the set of quintuples $(B,A,\eta,D,C)$, where
$(B,A)$ is a section of $G$, $(D,C)$ is a section of $H$, and $\eta:D/C\myiso B/A$
is a group isomorphism; the inverse of this bijection maps a quintuple
$(B,A,\eta,D,C)$ to the group $\{(g,h)\in B\times D\mid \eta(hC)=gA\}$.
Since $L/(k_1(L)\times k_2(L))\cong p_i(L)/k_i(L)$, for $i=1,2$, one has
\begin{equation}\label{eqn |L|}
  |L| =|p_1(L)|\cdot|k_2(L)| = |k_1(L)|\cdot|p_2(L)|\,.
\end{equation}
In the sequel we will freely identify every subgroup $L\leq G\times H$ with
its corresponding quintuple $(p_1(L),k_1(L),\eta_L,p_2(L),k_2(L))$. The common
isomorphism class of the groups $p_1(L)/k_1(L)$ and $p_2(L)/k_2(L)$ will be 
denoted by $q(L)$.
If $\alpha$ is an automorphism of $G$ then we abbreviate the
group $(G,1,\alpha,G,1)\leq G\times G$ by $\Delta_\alpha(G)$; if $\alpha=c_g$ for some
$g\in G$ then we just set $\Delta_g(G):=\Delta_\alpha(G)$, and if $g=1$ then we set
$\Delta(G):=\Delta_1(G)$.

\smallskip
(d)\, For subgroups $L\leq G\times H$ and $M\leq H\times K$, we set
$$L*M:=\{(g,k)\in G\times K\mid\exists\, h\in H: (g,h)\in L,\, (h,k)\in M\},$$
which is a subgroup of $G\times K$.
\end{nothing}

\begin{nothing}\label{noth bisets}{\bf Bisets and tensor products.}\:
(a)\, By a {\it $(G,H)$-biset} we will always understand a finite set, equipped
with a left $G$-action and a right $H$-action that commute with each other. 
The $(G,H)$-bisets form a category whose morphisms are the $(G,H)$-equivariant
set maps.
As usual, we will
freely identify $(G,H)$-bisets with left $G\times H$-sets, by setting
$$(g,h) x:=gxh^{-1}\quad\text{ and }\quad gyh:=(g,h^{-1})y,$$
for every $(G,H)$-biset $X$, every left $G\times H$-set $Y$, $g\in G$, $h\in H$,
$x\in X$, $y\in Y$.
In particular, for every $(G,H)$-biset $X$ and every $L\leq G\times H$, one can
define the {\it $L$-fixed points} $X^L$ of $X$.

\smallskip
(b)\, For every $(G,H)$-biset $X$, we denote by $X^\circ$ its {\it opposite biset}, which
an $(H,G)$-biset: as a set $X^\circ$ equals $X$, and the biset structure on $X^\circ$ is given
by $hx^\circ g:=(g^{-1}xh^{-1})$, where $g\in G$, $h\in H$, $x\in X$, and $x^\circ$ is $x$
viewed as an element in $X^\circ$. Note that $X=(X^\circ)^\circ$ as $(G,H)$-bisets.

On the other hand, one defines, for every subgroup $L\leq G\times H$,
its {\it opposite group} 
$$L^\circ:=\{(h,g)\in H\times G\mid (g,h)\in L\}\,,$$
which is clearly a subgroup of $H\times G$. 
If $M\leq H\times K$ then $(L*M)^\circ=M^\circ*L^\circ$.
Moreover, one has an isomorphism
of $(H,G)$-bisets
$$(H\times G)/L^\circ\to (G\times H/L)^\circ,\quad (h,g)L^\circ\mapsto ((g,h)L)^\circ.$$

\smallskip
(c)\, Suppose that $X$ is a $(G,H)$-biset and that $Y$ is an $(H,K)$-biset. Then their cartesian product
$X\times Y$ becomes a $(G,K)$-biset in the obvious way. Moreover, $X\times Y$ is also a left
$H$-set via $h(x,y):=(xh^{-1},hy)$, for $h\in H$, $x\in X$, $y\in Y$. This $H$-action
commutes with the $G\times K$-action, and the set $X\times_H Y$ of $H$-orbits on $X\times Y$
inherits the $(G,K)$-biset structure. We call $X\times_H Y$ the
{\it tensor product} of $X$ and $Y$ and denote the $H$-orbit of an element $(x,y)\in X\times Y$ by $x\times_H y\in X\times _H Y$.
\end{nothing}

\begin{nothing}\label{noth Burnside group}{\bf (Double) Burnside rings.}\,
(a)\, The {\em Burnside ring} $B(G)$ of a finite group $G$ is the Grothendieck ring of the category of finite left $G$-sets with respect to disjoint unions and direct products. The element in $B(G)$ associated to a finite left $G$-set $X$ will be denoted by $[X]$. Thus, if $\scrStilde_G\subseteq\scrS_G$ is a set of representatives of the conjugacy classes of subgroups of $G$ then the elements $[G/U]$, $U\in\scrStilde_G$, form a $\ZZ$-basis of $B(G)$, the {\em standard basis} of $B(G)$.

\smallskip
(b)\, The Grothendieck group of the category of $(G,H)$-bisets will be denoted
by $B(G,H)$, and is called the {\it double Burnside group} of $G$ and $H$.
Identifying $(G,H)$-bisets with left $G\times H$-sets as in \ref{noth bisets}(a), we may
also identify $B(G,H)$ with the classical Burnside group $B(G\times H)$. As in Part~(a), for every $(G,H)$-biset
$X$, the image of its isomorphism class in $B(G,H)$ will be denoted by $[X]$, and if $\tilde{\scrS}_{G\times H}\subseteq\scrS_{G\times H}$ denotes a transversal for the conjugacy classes of subgroup of $G\times H$ then the
elements $[G\times H/L]$, $L\in \tilde{\scrS}_{G\times H}$, form a $\mathbb{Z}$-basis of $B(G,H)$.

\smallskip
(c)\, The tensor product of bisets gives rise to a $\ZZ$-bilinear map
$$-\cdot_H -:B(G,H)\times B(H,K)\to B(G,K),\; ([X],[Y])\mapsto [X\times_H Y],$$
where $X$ is a $(G,H)$-biset and $Y$ is an $(H,K)$-biset. In the case where $G=H=K$
this defines a multiplication, turning $B(G,G)$ into a ring with identity element
$[G]=[G\times G/\Delta(G)]$; we call $B(G,G)$ the {\it double Burnside ring of $G$}. 

\smallskip
(d)\, Sending each $(G,H)$-biset to its opposite biset induces a group isomorphism
$$-^\circ:B(G,H)\to B(H,G),\; [X]\mapsto [X^\circ]$$
with
$$([X]\cdot_H[Y])^\circ=[Y]^\circ\cdot_H[X]^\circ\in B(K,G)\quad\text{ and }\quad ([X]^\circ)^\circ=[X]\in B(G,H),$$
where $X$ is a $(G,H)$-biset and $Y$ is an $(H,K)$-biset. Thus, in the case where $G=H=K$,
this defines an anti-involution of the ring $B(G,G)$.
\end{nothing}

The following Mackey-type formula shows how to express the tensor product of 
a standard basis element of $B(G,H)$ and a standard basis element of $B(H,K)$ as
a sum of standard basis elements of $B(G,K)$.

\begin{proposition}[\cite{BoucSLN}, 2.3.34]\label{prop Mackey}
For $L\leq G\times H$ and $M\leq H\times K$, one
has
$$[G\times H/L]\cdot_H [H\times K/M]=\sum_{h\in[p_2(L)\backslash H/p_1(M)]}[G\times K/(L*{}^{(h,1)}M)]\in B(G,K),$$
where $[p_2(L)\backslash H/p_1(M)]\subseteq H$ denotes a set of double coset representatives.
\end{proposition}

In order to analyze the map
$-\cdot_H-:B(G,H)\times B(H,K)\to B(G,K)$ in more detail,
it will, therefore, be important to examine the $*$-product of subgroups of $G\times H$ with
subgroups of $H\times K$. 

\smallskip
The first of the following two lemmas is a classical result due to Zassenhaus. The second one is most likely well known to experts; however, we have not been able to find a 
suitable reference, and therefore include a proof for the reader's 
convenience.

\begin{lemma}[Butterfly Lemma, \cite{Hupp} Hilfssatz I.11.3]\label{lem butter}
Let $(B,A)$ and $(D,C)$ be two sections of $G$. Then there exists a canonical isomorphism 
$$\beta(B',A';D',C')\colon D'/C'\to B'/A'\,,$$ 
where $A\le A'\unlhd B'\le B$ and 
$C\le C'\unlhd D'\le D$ are defined as
\begin{equation*}
  B':=(B\cap D)A\,,\quad A':=(B\cap C)A\,,\quad D':=(D\cap B)C\,,\quad\text{and}\quad C':=(D\cap A)C\,.
\end{equation*}
The isomorphism $\beta(B',A';D',C')$ is uniquely determined by the property 
that it maps $xC'$ to $xA'$ for every $x\in B\cap D$.
\end{lemma}

The following definition is due to Bouc, cf.~\cite[Definition~4.3.12]{BoucSLN}.

\begin{definition}\label{def linked}
Sections $(B,A)$ and $(D,C)$ of $G$ are called {\it linked} if 
$$D\cap A=C\cap B,\quad (D\cap B)C=D,\quad \text{ and }\quad (D\cap B)A=B;$$
in this case we write $(D,C) \link (B,A)$.
\end{definition}

\begin{lemma}\label{lem butterfly}
Let $L=(P_1,K_1,\phi,P_2,K_2)\le G\times H$ and $M=(P_3,K_3,\psi,P_4,K_4)\le H\times K$.
Then
\begin{equation*}
  L*M = (P'_1,K'_1,\bar{\phi}\circ \beta(P'_2,K'_2;P'_3,K'_3)\circ \bar{\psi}, P'_4,K'_4)\,,
\end{equation*}
where
\begin{itemize}
\item $K_2\le K'_2\le P'_2\le P_2$ and $K_3\le K'_3\le P'_3\le P_3$ are determined by 
the Butterfly Lemma applied to the sections $(P_2,K_2)$ and $(P_3,K_3)$ of $H$;
\item $K_1\le K'_1\le P'_1\le P_1$ and $K_4\le K'_4\le P'_4\le P_4$ are determined by
\begin{gather*}
  P'_1/K_1=\phi(P'_2/K_2)\,,\quad K'_1/K_1=\phi(K'_2/K_2)\,, \\
  P'_4/K_4=\psi^{-1}(P'_3/K_3)\,,\quad K'_4/K_4=\psi^{-1}(K'_3/K_3)\,;
\end{gather*}
\item the isomorphisms $\bar{\phi}\colon P'_2/K'_2\to P'_1/K'_1$ and $\bar{\psi}\colon P'_4/K'_4\to P'_3/K'_3$ 
are induced by the isomorphisms $\phi$ and $\psi$.
\end{itemize}

\noindent
In particular, 

{\rm (i)}\, if $P_2=P_3$ then $(P_2',K_2')=(P_2,K_2K_3)=(P_3',K_3')$ and
$$L*M=(P_1,K_1',\bar{\phi}\circ\bar{\psi},P_4,K_4'),$$
where $K_1'$ and $K_4'$ are such that $K_1'/K_1=\phi(K_2K_3/K_2)$ and
$K_4'/K_4=\psi^{-1}(K_2K_3/K_3$);

{\rm (ii)}\, if $(P_2, K_2) = (P_3,K_3)$ then 
\begin{equation*}
  L*M = (P_1,K_1,\phi\circ\psi,P_4,K_4).
\end{equation*}
\end{lemma}

The following diagram illustrates the result of the lemma:

\bigskip
\bigskip
\bigskip

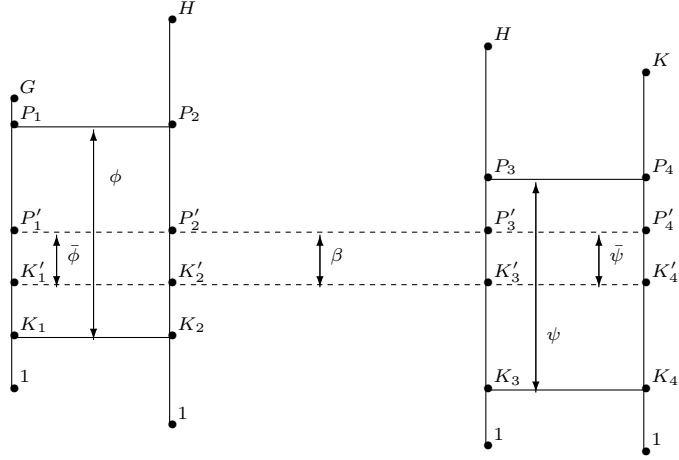
\begin{center}

\unitlength 7mm
{\scriptsize
\begin{picture}(13,6)
\put(0,5.5){$\bullet$} \put(0,0){$\bullet$}
\put(0.2,5.7){$G$}      \put(0.2,0.2){$1$}
\put(0,5){$\bullet$} \put(0,1){$\bullet$}
\put(0.2,5.2){$P_1$}   \put(0.2,1.2){$K_1$}
\put(0.05,5.05){\line(1,0){3}}
\put(0.05,1.05){\line(1,0){3}}
\put(1.6,5){\vector(0,-20){4}}
\put(1.6,1){\vector(0,20){4}}
\put(1.9,4){$\phi$}
\put(0.05,5.55){\line(0,-1){5.5}}

\put(0,3){$\bullet$} \put(0,2){$\bullet$}
\put(0.2,3.2){$P_1'$}  \put(0.2,2.2){$K_1'$}
\multiput(0.05,3.05)(0.2,0){60}{\line(1,0){0.1}}
\multiput(0.05,2.05)(0.2,0){60}{\line(1,0){0.1}}
\put(0.9,3){\vector(0,-1){1}}
\put(0.9,2){\vector(0,1){1}}
\put(1.1,2.5){$\bar{\phi}$}

\put(3,7){$\bullet$}  \put(3,-0.7){$\bullet$}
\put(3.2,7.2){$H$}      \put(3.2,-0.5){$1$}
\put(3,5){$\bullet$} \put(3,1){$\bullet$}
\put(3.2,5.2){$P_2$}   \put(3.2,1.2){$K_2$}
\put(3,3){$\bullet$} \put(3,2){$\bullet$}
\put(3.2,3.2){$P_2'$}  \put(3.2,2.2){$K_2'$}
\put(5.9,3){\vector(0,-1){1}}
\put(5.9,2){\vector(0,1){1}}
\put(6.1,2.5){$\beta$}
\put(3.05,7.05){\line(0,-1){7.7}}

\put(9,6.5){$\bullet$}  \put(9,-1.1){$\bullet$}
\put(9.2,6.7){$H$}      \put(9.2,-0.9){$1$}
\put(9,4){$\bullet$} \put(9,0){$\bullet$}
\put(9.2,4.2){$P_3$}   \put(9.2,0.2){$K_3$}
\put(9,3){$\bullet$} \put(9,2){$\bullet$}
\put(9.2,3.2){$P_3'$}   \put(9.2,2.2){$K_3'$}
\put(9.05,4.05){\line(1,0){3}}
\put(9.05,0.05){\line(1,0){3}}
\put(11.2,3){\vector(0,-1){1}}
\put(11.2,2){\vector(0,1){1}}
\put(11.4,2.5){$\bar{\psi}$}
\put(10,4){\vector(0,-1){4}}
\put(10,0){\vector(0,1){4}}
\put(10.2,1){$\psi$}
\put(9.05,6.55){\line(0,-1){7.6}}

\put(12,6){$\bullet$}  \put(12,-1,2){$\bullet$}
\put(12.2,6.2){$K$}     \put(12.2,-1){$1$}
\put(12,4){$\bullet$}\put(12,0){$\bullet$}
\put(12.2,4.2){$P_4$}   \put(12.2,0.2){$K_4$}
\put(12,3){$\bullet$} \put(12,2){$\bullet$}
\put(12.2,3.2){$P_4'$}   \put(12.2,2.2){$K_4'$}
\put(12.05,6.05){\line(0,-1){7.2}}

\end{picture}}

\end{center}

\bigskip
\bigskip
\bigskip

\begin{proof}
For $g\in G$ one has
\begin{align*}
  g\in p_1(L*M) & \iff \exists h\in H, k\in K: (g,h)\in L, (h,k)\in M \\
  & \iff \exists h\in P_2\cap P_3: (g,h)\in L \\
  & \iff gK_1\in\phi((P_2\cap P_3)K_2/K_2)\,.
\end{align*}
This implies that $p_1(L*M)=P'_1$. Similarly one shows that $p_2(L*M)=P'_4$. Next, for $g\in G$ one has
\begin{align*}
  g\in k_1(L*M) & \iff \exists h\in H\colon (g,h)\in L, (h,1)\in M \\
  & \iff \exists h\in k_1(M): (g,h)\in L \\
  & \iff gK_1 \in \phi((K_3\cap P_2)K_2/K_2)\,.
\end{align*}
This implies that $k_1(L*M)=K'_1$. Similarly one shows that $k_2(L*M)=K'_4$. Finally, for $(g,k)\in G\times K$ one has
\begin{align*}
  (g,k)\in L*M & \iff \exists h\in H: (g,h)\in L, (h,k)\in M \\
  & \iff \exists h\in P_2\cap P_3: gK_1=\phi(hK_2), hK_3=\psi(kK_4) \\
  & \ \, \Longrightarrow \exists h\in P_2\cap P_3: gK'_1 = \bar{\phi}(hK'_2), hK'_3 = \bar{\psi}(kK'_4) \\
  & \ \, \Longrightarrow gK'_1 = \bigl(\bar{\phi}\circ \beta(P'_2,K'_2;P'_3,K'_3)\circ\bar{\psi}\bigr)(kK'_4)\,.
\end{align*}
This implies that the isomorphism $\eta_{L*M}\colon p_2(L*M)/k_2(L*M)\to p_1(L*M)/k_1(L*M)$ is equal to 
$\bar{\phi}\circ \beta(P'_2,K'_2;P'_3,K'_3)\circ\bar{\psi}$.
\end{proof}


\section{Natural embeddings of the (double) Burnside ring}\label{sec embeddings}

\noindent
Recall from \ref{noth groups} that, for every finite group $G$, we denote by 
$\scrS_G$ the set of subgroups of $G$, and by $\tilde{\scrS}_G\subseteq \scrS_G$ a transversal
for the conjugacy classes of $\scrS_G$.
Moreover, for $U\leq G$, we denote by $[U]_G$ the $G$-conjugacy 
class of $U$. If $R$ is a commutative ring then 
we denote by $[U]_G^+\in R\scrS_G$ the sum of the elements in $[U]_G$; here, 
$R\scrS_G$ denotes the free $R$-module with basis $\scrS_G$. Since $G$ acts on 
$\scrS_G$ by conjugation, we can view $R\scrS_G$ as a permutation $RG$-module. Its 
fixed points are denoted by $(R\scrS_G)^G$. We use the same notation $X^G$ for the 
set of $G$-fixed points on any $G$-set $X$.

Recall that the {\em M\"obius function} $\scrX\times\scrX\to \ZZ$, $(x,y)\mapsto \mu^\scrX_{x,y}$, of a finite partially ordered set (poset), $(\scrX,\le)$ is defined by setting $\mu^{\scrX}_{x,y}=0$ if $x\not\le y$, and in the case that $x\le y$ the integer $\mu^{\scrX}_{x,y}$ is defined recursively by the equation $\sum_{x\le z\le y} \mu_{x,z}^{\scrX} = \delta_{x,y}$, where $\delta_{x,y}$ denotes the Kronecker symbol.

\begin{definition}\label{def alpha,zeta,rho}
We define additive maps
\begin{alignat*}2
  \alpha_G\colon B(G) & \to \ZZ\scrS_G\,, & \quad [X] & \mapsto \sum_{x\in X} \stab_G(x)\,, \\ 
  \rho_G\colon B(G) & \to \ZZ\scrS_G\,, & \quad [X] & \mapsto \sum_{U\le G} |X^U|\cdot U\,, \\
  \zeta_G\colon \ZZ\scrS_G & \to \ZZ\scrS_G\,, & \quad U & \mapsto \sum_{U'\le U} U'\,.
\end{alignat*}
The map $\rho_G$ is the well-studied classical {\em mark homomorphism}.
Clearly, $\zeta_G$ is an isomorphism with inverse $\mu_G$ given by $\mu_G(U)=\sum_{U'\le U}\mu_{U',U}\cdot U'$, 
where $\mu_{U',U}$ denotes the M\"obius function with respect to the poset $\scrS_G$. Extending scalars 
from $\ZZ$ to $R$ one obtains $R$-module homomorphisms $\alpha_G\colon RB(G)\to R\scrS_G$, 
$\zeta_G\colon R\scrS_G\to R\scrS_G$ and $\rho_G\colon RB(G)\to R\scrS_G$. 
Then $\zeta_G\colon R\scrS_G\to R\scrS_G$ 
is an isomorphism of $RG$-modules and induces an isomorphism
\begin{equation*}
  \zeta_G\colon (R\scrS_G)^G\to (R\scrS_G)^G\,.
\end{equation*}
\end{definition}

\begin{proposition}\label{prop B(G)}
Let $G$ be a finite group and let $R$ be a commutative ring.

\smallskip
{\rm (a)} For $U\le G$, one has $\alpha_G([G/U])=[N_G(U):U] \cdot [U]_G^+$. 
In particular, if $R$ has trivial $|G|$-torsion then $\alpha_G\colon RB(G)\to R\scrS_G$ is injective.

\smallskip
{\rm (b)} One has $\zeta_G\circ\alpha_G=\rho_G$. 

\smallskip
{\rm (c)} The images of $\alpha_G$ and $\rho_G$ are contained in $(R\scrS_G)^G$.  
In particular, one obtains a commutative diagram
\begin{diagram}[90]
  RB(G) & \Ear{\alpha_G} & (R\scrS_G)^G & \subseteq & R\scrS_G &&
        & \seaR{\rho_G} & \saR{\zeta_G} &            & \saR{\zeta_G} &&
        &               & (R\scrS_G)^G & \subseteq  & R\scrS_G &&
\end{diagram}
of $R$-module homomorphisms in which both vertical maps $\zeta_G$ are 
isomorphisms. If $|G|$ is invertible in $R$ then also $\alpha_G$ and $\rho_G$ 
in the above diagram are isomorphisms.
For $R=\ZZ$ one has
\begin{equation*}
  [(\ZZ\scrS_G)^G:\alpha_G(B(G))] = \prod_{U\in\scrStilde_G} [N_G(U):U] = [(\ZZ\scrS_G)^G:\rho_G(B(G))].\,
\end{equation*}
\end{proposition}

\begin{proof}
(a) The stabilizer subgroup of $gU$ is equal to $\gU$ and it occurs as the 
stabilizer of precisely $[N_G(U):U]$ elements in $G/U$, namely the elements 
$gnU$ for $n\in N_G(U)$.

\smallskip
(b) For any finite $G$-set $X$ one has
\begin{equation*}
  \zeta_G(\alpha_G([X])) = \sum_{x\in X} \sum_{K\le\stab_G(x)} K =
  \sum_{\substack{(K,x)\in \scrS_G\times X \\ x\in X^K}} K = \sum_{K\le G}|X^K| \cdot K  = \rho_G([X])\,.
\end{equation*}

\smallskip
(c) By Part~(a), the image of $\alpha_G$ is contained in $(R\scrS_G)^G$ and, 
since $\zeta_G$ is an isomorphism of $RG$-modules, Part~(b) implies that 
$\rho_G(RB(G))\subseteq (R\scrS_G)^G$. 
The last two statements follow from Part~(a) 
and the fact that the elements $[U]_G^+$, $U\in\scrStilde_G$, form an $R$-basis 
of $(R\scrS_G)^G$.
\end{proof}

\begin{remark}
Let $G$ and $H$ be finite groups and let $R$ be a commutative ring.

(a) Although we will not make use of the following fact in this paper it seems 
worth to point it out. One can endow $\ZZ\scrS_G$ with two ring structures. 
The first multiplication is given by $(U,V)\mapsto U\cap V$ and the second one by 
$(U,V)\mapsto \delta_{U,V}U$, for $U,V\in\scrS_G$. Extending scalars to 
$R$ we obtain two $R$-algebra structures on $R\scrS_G$. It is easy to verify that 
the $G$-fixed points $(R\scrS_G)^G$ form a subalgebra for both $R$-algebra 
structures. Moreover, it is easy to verify that all maps in the commutative diagram in 
Proposition~\ref{prop B(G)}(c) are $R$-algebra homomorphisms if one endows $RB(G)$ 
with the usual multiplication coming from the ring structure of the Burnside 
ring $B(G)$, and if one endows $R\scrS_G$ in the top and bottom row with the 
first and second multiplication, respectively.

\smallskip
(b) From Proposition~\ref{prop B(G)}(c) we obtain a commutative diagram
\begin{diagram}
  \movevertex(-10,0){RB(G,H)} & \movearrow(-10,0){\Ear[25]{\alpha_{G,H}}} & (R\scrS_{G\times H})^{G\times H} & \subseteq & R\scrS_{G\times H} &&
  & \seaR{\rho_{G,H}} & \saR{\zeta_{G,H}} & & \saR{\zeta_{G,H}} &&
  & & (R\scrS_{G\times H})^{G\times H} & \subseteq & R\scrS_{G\times H} &&
\end{diagram}
by replacing $G$ with $G\times H$ and slightly renaming the maps. 
If also $K$ is a finite group we will next define a multiplication map 
$R\scrS_{G\times H}\times R\scrS_{H\times K}\to R\scrS_{G\times K}$ that is compatible 
with the tensor product construction $RB(G, H)\times RB(H,K)\to RB(G,K)$ via the embedding $\alpha$.
\end{remark}

\begin{definition}\label{def *_H}
Let $G$, $H$, and $K$ be finite groups, and let $R$ be a commutative ring such that $|H|$ is invertible in $R$.
For $L\le G\times H$ and $M\le H\times K$ we set
\begin{equation*}
  \kappa(L,M):=\frac{|k_2(L)\cap k_1(M)|}{|H|}\,\in\, R
\end{equation*}
and we define the $R$-bilinear map
\begin{equation*}
  - *_H^\kappa-\colon R\scrS_{G\times H} \times R\scrS_{H\times K} \to R\scrS_{G\times K}
\end{equation*}
by
\begin{equation*}
  L*_H^\kappa M := \kappa(L,M) \cdot (L*M)\,.
\end{equation*}
\end{definition}

\begin{proposition}\label{prop cocycle}
Let $G$, $H$, $K$, and $I$ be finite groups, and let $R$ be a commutative ring such 
that $|H|$ and $|K|$ are invertible in $R$. For any $L\le G\times H$, 
$M\le H\times K$, and $N\le K\times I$, the $2$-cocycle relation
\begin{equation*}
  \kappa(L,M)\cdot\kappa(L*M,N) =\kappa(L,M*N)\cdot\kappa(M,N)
\end{equation*}
holds in $R^\times$. In particular, the bilinear maps defined in 
Definition~\ref{def *_H} behave associatively:
\begin{equation*}
  (L *_H^\kappa M) *_K^\kappa N = L *_H^\kappa (M *_K^\kappa N)\,.
\end{equation*}
Moreover, if $|G|$ is invertible in $R$ then the $R$-module $R\scrS_{G\times G}$ 
together with the multiplication $*_G^\kappa$ is an $R$-algebra with identity element 
$|G|\cdot\Delta(G)$.
\end{proposition}

\begin{proof}
It suffices to show that 
\begin{equation*}
  |k_2(L)\cap k_1(M)|\cdot|k_2(L*M)\cap k_1(N)| = |k_2(L)\cap k_1(M*N)|\cdot|k_2(M)\cap k_1(N)|\,
\end{equation*}
holds in $\ZZ$. To show this it suffices to show that one has a group isomorphism
\begin{equation*}
  \frac{k_2(L)\cap k_1(M*N)}{k_2(L)\cap k_1(M)} \cong \frac{k_2(L*M)\cap k_1(N)}{k_2(M)\cap k_1(N)}\,.
\end{equation*}
In order to construct a homomorphism between these two groups, let 
$h\in k_2(L)\cap k_1(M*N)$. Then $(1,h)\in L$ and $(h,1)\in M*N$. This 
implies that there exists $k\in K$ such that $(h,k)\in M$ and $(k,1)\in N$. 
Thus, $k\in k_2(L*M)\cap k_1(N)$. If also $k'\in k_1(N)$ is such that $(h,k')\in M$ 
then $k^{-1}k'\in k_2(M)\cap k_1(N)$. Therefore we have a well-defined function 
that maps the class of $h$ to the class of $k$, with $(h,k)\in M$ and $(k,1)\in N$. 
Symmetrically, one has a well-defined function in the other direction that 
maps the class of an element $k\in k_2(L*M)\cap k_1(N)$ to the class of $h$, where 
$(h,k)\in M$ and $(1,h)\in L$. Clearly these two functions are inverses of each 
other. It is also easy to see that they are group homomorphisms.

The last two statements are immediate consequences of the $2$-cocycle relation.
\end{proof}

\begin{proposition}\label{prop alpha compatible}
Let $G$, $H$, and $K$ be finite groups, and let $R$ be a commutative ring such that 
$|H|$ is invertible in $R$. Moreover let $a\in RB(G,H)$ and $b\in RB(H,K)$. Then
\begin{equation*}
  \alpha_{G,K}(a \cdot_H b) = \alpha_{G,H}(a) *_H^\kappa \alpha_{H,K}(b)\,.
\end{equation*}
\end{proposition}

\begin{proof}
We may assume that $a=[X]$ and $b=[Y]$ for a finite $(G,H)$-biset $X$ and a finite 
$(H,K)$-biset $Y$.
Then we have
\begin{align*}
  \alpha_{G,K}([X]\cdot_H[Y]) & = \alpha_{G,K}([X\times_H Y]) = 
  \sum_{x \times\!_H y\,\in X\times_H Y} \stab_{G\times K} (x\times_H y) \\
  & = \sum_{x \times\!_H\, y\in X\times_H Y} \stab_{G\times H}(x) * \stab_{H\times K}(y) \\
  & = \sum_{(x,y)\in X\times Y} \frac{|k_2(\stab_{G\times H}(x))\cap k_1(\stab_{H\times K}(y))|}{|H|} \cdot
                 \stab_{G\times H}(x) * \stab_{H\times K}(y) \\
  & = \alpha_{G,H}([X]) *_H^\kappa \alpha_{H,K}([Y])\,.
\end{align*}

\end{proof}

For any finite group $G$ and any commutative ring $R$ we set
\begin{equation*}
  e_G:= \sum_{g\in G} \Delta_g(G) = |Z(G)|\cdot \sum_{c\in\Inn(G)} \Delta_c(G) \ \in R\scrS_{G\times G}\,.
\end{equation*}
In the case where $|G|$ is invertible in $R$ it 
follows immediately from the definition of the bilinear map $-*_H^\kappa-$ in Definition~\ref{def *_H} that $e_G$ is an idempotent in the $R$-algebra $(R\scrS_{G\times G},*_G^\kappa)$.

\begin{proposition}\label{prop fixed points}
Let $R$ be a commutative ring, and let $G$ and $H$ be finite groups whose orders are invertible in $R$.

\smallskip
{\rm (a)} One has 
\begin{equation*}
  \alpha_{G,H}(RB(G,H)) = (R\scrS_{G\times H})^{G\times H} = 
  e_G *_G^\kappa (R\scrS_{G\times H}) *_H^\kappa e_H\,.
\end{equation*}

\smallskip
 {\rm (b)} The double Burnside algebra $RB(G,G)$ is isomorphic to $e_G *_G^\kappa (R\scrS_{G\times G})*_G^\kappa e_G$.
\end{proposition}

\begin{proof}
(a) The first equation was already proved in Proposition~\ref{prop B(G)}(c). For the proof of 
the second equation note that, for $g\in G$, $h\in H$ and $L\le G\times H$, one has
\begin{equation*}
  \Delta_g(G) * L * \Delta_h(H) = \lexp{(g,h^{-1})}{L}\,,
\end{equation*}
by Lemma~\ref{lem butterfly}. This implies that 
\begin{equation*}
  e_G *_G^\kappa L *_H^\kappa e_H = \frac{|N_{G\times H}(L)|}{|G\times H|} \cdot [L]_{G\times H}^+\,. 
\end{equation*}
Now the second equation is immediate.

\smallskip
(b) This is an immediate consequence of Proposition~\ref{prop alpha compatible} and Part~(a), 
since $\alpha_{G,G}$ is injective.
\end{proof}

Finally, we will show in the next proposition that the bilinear map $- *_H^\kappa -$ from
Definition~\ref{def *_H} is translated via the isomorphism $\zeta$ into the following bilinear map.

\begin{definition}\label{def aLMN}
Let $G$, $H$, and $K$ be finite groups, and let $R$ be a commutative ring such that $|H|$ is 
invertible in $R$. We define the $R$-bilinear map
\begin{equation*}
  - \startilde_H^\kappa - \colon R\scrS_{G\times H} \times R\scrS_{H\times K} \to R\scrS_{G\times K}\,, \quad
  (L,M)\mapsto \sum_{N\le G\times K} a^{L,M}_N \cdot N\,,
\end{equation*}
where, for $L\le G\times H$, $M\le H\times K$, and $N\le G\times K$, the coefficient $a_N^{L,M}\in R$ is defined by
\begin{equation*}
  a^{L,M}_N :=\frac{1}{|H|} \sum_{\substack{(L',M')\le (L,M)\\N\le L'*M'}} 
  \mu_{(L',M'),(L,M)}^{\scrS_{G\times H}\times\scrS_{H\times K}} \cdot |k_2(L')\cap k_1(M')| 
  = \sum_{\substack{(L',M')\in\scrS_{G\times H}\times\scrS_{H\times K}\\ N\le L'*M'}} 
  \mu_{(L',M'),(L,M)}^{\scrS_{G\times H}\times\scrS_{H\times K}} \cdot \kappa(L',M')\,.
\end{equation*}
Here, $\mu_{(L',M'),(L,M)}^{\scrS_{G\times H}\times\scrS_{H\times K}}$ denotes the M\"obius 
function on the poset $\scrS_{G\times H}\times \scrS_{H\times K}$ equipped with 
the direct product partial order on $\scrS_{G\times H}\times \scrS_{H\times K}$, 
i.e., $(L',M')\le (L,M)$ if and only if $L'\le L$ and $M'\le M$. Note that 
$a^{M^\circ,L^\circ}_{N^\circ} =a^{L,M}_N$ and 
$(a\startilde_H^\kappa b)^\circ = b^\circ \startilde_H^\kappa a^\circ$ 
for $a\in R\scrS_{G\times H}$ and $b\in R\scrS_{H\times K}$. Here, $-^\circ\colon R\scrS_{G\times H}\to R\scrS_{H\times G}$ is defined as the $R$-linear extension of the map 
$\scrS_{G\times H}\to \scrS_{H\times G},\; L\mapsto L^\circ$.
\end{definition}

\begin{proposition}\label{prop zeta}
Let $G$, $H$, and $K$ be finite groups, and let $R$ be a commutative ring such 
that $|H|$ is invertible in $R$.

\smallskip
{\rm (a)} For all $a\in R\scrS_{G\times H}$ and $b\in R\scrS_{H\times K}$ one has 
$\zeta_{G,K}(a*_H^\kappa b) = \zeta_{G,H}(a)\ \startilde_H^\kappa\ \zeta_{H,K}(b)$.

\smallskip
{\rm (b)} If $|G|$ is invertible in $R$ then the map $\rho_{G,G}$ defines an 
isomorphism between the double Burnside $R$-algebra $RB(G,G)$ and the subring 
\begin{equation*}
  (R\scrS_{G\times G})^{G\times G}= \etilde_G\ \startilde_G^\kappa\ R\scrS_{G\times G}\ 
  \startilde_G^\kappa \ \etilde_G\end{equation*} 
of $(R\scrS_{G\times G}, \startilde_G^\kappa)$,
where $\etilde_G:=\zeta_G(e_G)$ is an idempotent in $(R\scrS_{G\times G},\startilde_G^\kappa)$.
\end{proposition}

\begin{proof}
(a) Recalling from Definition~\ref{def alpha,zeta,rho} that the maps $\zeta$ are 
isomorphisms with inverse $\mu$, it suffices to prove the equation for $a=\mu_{G,H}(L)$ 
and $b=\mu_{H,K}(M)$ with $L\le G\times H$ and $M\le H\times K$. In this case we have
\begin{align*}
  \zeta_{G,K}(a *_H^\kappa b) & = \sum_{L'\le L} \sum_{M'\le M} \mu_{L',L}\cdot\mu_{M',M}\cdot \zeta_{G,K}(L' *_H^\kappa M') \\
  & =  \sum_{(L',M')\le (L,M)} \mu_{(L',M'),(L,M)}^{\scrS_{G\times H}\times\scrS_{H\times K}}
  \cdot\kappa(L',M') \cdot \sum_{N\le L'*M'} N \\
  & = \sum_{N\le G\times K}\Bigl( \sum_{\substack{(L',M')\le (L,M)\\N\le L'*M'}}
  \mu_{(L',M'),(L,M)}^{\scrS_{G\times H}\times\scrS_{H\times K}} \cdot \kappa(L',M')\Bigr) \cdot N \\
  & = \sum_{N\le G\times K} a^{L,M}_N \cdot N = L\startilde_H^\kappa M = \zeta_{G\times H}(a)\ \startilde_H^\kappa\ \zeta_{H\times K}(b)\,.
\end{align*}

\smallskip
(b) This follows immediately from Part~(a) and Proposition~\ref{prop fixed points}.
\end{proof}


\section{Simplifying the coefficient $a^{L,M}_N$}\label{sec aLMN}

Throughout this section, $G$, $H$, and $K$ denote finite groups, and $R$ denotes a 
commutative ring such that $|H|$ is invertible in $R$. Moreover, we will consider subgroups
$L\le G\times H$, $M\le H\times K$, and $N\le G\times K$. The goal of this section 
is a simplification of the formula for the `multiplication' structure constants $a^{L,M}_N$ in 
Definition~\ref{def aLMN}. 
For this we will need to work with the partially ordered set $\scrS_{G\times H}\times \scrS_{H\times K}$ endowed with 
the direct product poset structure of $\scrS_{G\times H}$ and $\scrS_{H\times K}$, cf.~Definition~\ref{def aLMN}.
For $N\le G\times K$, we set
\begin{equation*}
  \scrX_N=:= \{(L,M)\in\scrS_{G\times H}\times \scrS_{H\times K} 
  \mid N\le L*M\}
\end{equation*}
and consider it as a subposet of $\scrS_{G\times H}\times\scrS_{H\times K}$.
Note that for any two elements $(L',M')\le (L,M)$ in $\scrX_N$ one has
\begin{equation}\label{eqn Gamma consequence}
  p_1(N)\le p_1(L'*M')\le p_1(L')\le p_1(L)\quad\text{and}\quad
  p_2(N)\le p_2(L'*M')\le p_2(M')\le p_2(M)\,.
\end{equation}
Note also that $\scrX_N$ is a {\em convex} subposet of $\scrS_{G\times H}\times \scrS_{H\times K}$, i.e., for $x,y,z\in\scrS_{G\times H}\times\scrS_{H\times K}$ with $x\le y\le z$ and $x,z\in \scrX_N$ one also has $y\in\scrX_N$. Thus, we can rewrite $a^{L,M}_N$ from Definition~\ref{def aLMN} as
\begin{equation}\label{eqn aLMN}
  a^{L,M}_N = \sum_{(L',M')\in\scrX_N} \mu_{(L',M'),(L,M)}^{\scrX_N} \cdot \kappa(L',M')\,.
\end{equation}
We will simplify this sum using the following lemma.

\begin{lemma}\label{lem referee}
Let $(\scrX,\le)$ be a finite poset and let $A$ be an abelian group. Moreover, let $\phi\colon \scrX\to\scrX$ and $\kappa\colon\scrX\to A$ be maps such that

\smallskip
{\rm (i)} $\phi$ is a map of posets, $\phi\circ\phi=\phi$, and $\phi(x)\le x$ for all $x\in\scrX$, and

\smallskip
{\rm (ii)} $\kappa\circ\phi=\kappa\colon\scrX\to A$.

\smallskip\noindent
Set $\scrY:=\phi(\scrX)$. Then, for each $z\in\scrX$, one has
\begin{equation*}
  \sum_{x\in\scrX}\mu_{x,z}^{\scrX}\cdot \kappa(x) =
  \begin{cases} 0 & \text{if $z\notin\scrY$,}\\
                         \mathop{\sum}\limits_{y\in\scrY} \mu_{y,z}^{\scrY}\cdot \kappa(y) & \text{if $z\in\scrY$,}
  \end{cases}
\end{equation*}
where $\mu^{\scrX}$ denotes the M\"obius function of $\scrX$, and $\mu^{\scrY}$ denotes the M\"obius function of $\scrY$, considered as a subposet of $\scrX$.
\end{lemma}

\begin{proof} 
Consider the set $\calC_z$ of chains $\sigma=(x_0<\ldots <x_n=z)$
in $\scrX$, and for each such $\sigma\in\mathcal{C}_z$, set 
$\kappa(\sigma):=\kappa(x_0)$. Then we have
\begin{equation*}
  \sum_{x\in\scrX}\mu^{\scrX}_{x,z}\cdot\kappa(x)=
  \sum_{\sigma\in\mathcal{C}_z}(-1)^{|\sigma|}\cdot\kappa(\sigma)\,,
\end{equation*}
where $|\sigma|$ denotes the length of the chain $\sigma$, i.e., $|\sigma|=n$ for the chain $\sigma$ above.

We now denote by $\mathcal{D}_z\subseteq\mathcal{C}_z$ the set of those chains
all of whose terms belong to $\scrY$, and we 
define a map $\psi:\mathcal{C}_z\smallsetminus \mathcal{D}_z\to\mathcal{C}_z\smallsetminus \mathcal{D}_z$ such that
$$\psi\circ\psi=\id,\quad |\psi(\sigma)|-|\sigma|\in\{\pm 1\},\quad  \kappa(\psi(\sigma))=\kappa(\sigma)\,.$$
For this, let 
$\sigma\in\mathcal{C}_z\smallsetminus \mathcal{D}_z$ with $\sigma=(x_0<\ldots<x_n=z)$. Then there
is a minimal index $i\in\{0,\ldots,n\}$ such that $x_i\notin\scrY$, that is, 
$y_i:=\phi(x_i)<x_i$, since, by definition,
$\scrY:=\{x\in\scrX\mid \phi(x)=x\}$. If $i\geq 1$ and $x_{i-1}=y_i$ then we define
$\psi(\sigma)$ as the chain obtained from $\sigma$ by removing $x_{i-1}$,
and if $i=0$ or if $i\geq 1$ and $x_{i-1}<y_i$ then we define $\psi(\sigma)$
as the chain obtained from $\sigma$ by inserting $y_i$.
This yields the desired map $\psi$, and we thus get
\begin{equation*}
  \sum_{x\in\scrX}\mu^{\scrX}_{x,z}\cdot \kappa(x)=
  \sum_{\sigma\in\mathcal{C}_z}(-1)^{|\sigma|}\cdot\kappa(\sigma)=  
  \sum_{\sigma\in\mathcal{D}_z}(-1)^{|\sigma|}\cdot\kappa(\sigma)=
  \begin{cases}0&\text{ if } z\notin\scrY,\\
     \sum_{y\in\scrY}\mu^{\scrY}_{y,z}\cdot \kappa(y)&\text{ if } z\in\scrY\,.
  \end{cases}
\end{equation*}
\end{proof}

%
%
\bigskip

Let $N\le G\times K$. We will apply Lemma~\ref{lem referee} to the partially ordered set $\scrX:=\scrX_N$, to the map $\phi\colon\scrX\to\scrX$, $(L,M)\mapsto(\Ltilde,\Mtilde)$, where
\begin{gather*}
  \Ltilde:=\{(g,h)\in L\mid g\in p_1(N), \exists k\in p_2(N)\colon (h,k)\in M\}\,, \\
  \Mtilde:=\{(h,k)\in M \mid k\in p_2(N), \exists g\in p_1(N)\colon (g,h)\in L\}\,,
\end{gather*}
to the additive group of the ring $R$, and to the function $\kappa\colon \scrX\to R$, $(L,M)\mapsto |k_2(L)\cap k_1(M)|/|H|$ as in Definition~\ref{def *_H}. Note that for simplicity we used abusive notation: the subgroup $\Ltilde$ depends on $L$ and $M$, not just on $L$. Similarly, $\Mtilde$ depends on $L$ and $M$. It is now a straightforward verification that the conditions (i) and (ii) in Lemma~\ref{lem referee} are satisfied and that
\begin{equation*}
  \scrY_N:=\phi(\scrX_N) = \{(L,M)\in\scrX_N\mid p_1(L)=p_1(N),\ p_2(L)=p_2(M),\ p_2(M)=p_2(N)\}\,.
\end{equation*}
For $(L,M)\in\scrY_N$ we also set
\begin{equation*}
  \scrY_N^{L,M}:=\{(L',M')\in\scrY_N\mid (L',M')\le (L,M)\}\,.
\end{equation*}
The following theorem is now an immediate consequence of Lemma~\ref{lem referee} applied to Equation~(\ref{eqn aLMN}).

\begin{theorem}\label{thm aLMN}
Let $L\le G\times H$, $M\le H\times K$ and $N\le G\times K$. Then
\begin{equation*}
  a^{L,M}_N = \begin{cases} 0 & \text{if $(L,M)\notin\scrY_N$,} \\
         \mathop{\sum}\limits_{(L',M')\in\scrY_N^{L,M}} \mu^{\scrY_N}_{(L',M'),(L,M)}\cdot \kappa(L',M') 
         & \text{if $(L,M)\in\scrY_N$.}
   \end{cases}
\end{equation*}
\end{theorem}

From the above theorem and Definition~\ref{def aLMN} the following corollary is immediate.

\begin{corollary}\label{cor aLMN}
Let $L\le G\times H$ and $M\le H\times K$. Then $L\startilde_H^\kappa M$ is a linear combination of subgroups $N\le G\times K$ satisfying $p_1(N)=p_1(L)$ and $p_2(N)=p_2(M)$. Moreover, if $p_1(L)\neq p_2(M)$ then $L\startilde_H^\kappa M = 0$.
\end{corollary}

Next we consider the `multiplication' of special types of subgroups $L$ and $M$ with respect to $\startilde_H^\kappa$.

\begin{proposition}\label{prop left-free}
Assume that $k_1(L)=1$ and $k_1(M)=1$ (or that $k_2(L)=1$ and $k_2(M)=1$). Then
\begin{equation*}
  L\startilde^\kappa_H M = 
  \begin{cases}\frac{1}{|H|}\cdot L*M & \text{if $p_2(L)=p_1(M)$,}\\
       0 & \text{if $p_2(L)\neq p_1(M)$\,.}
  \end{cases}
\end{equation*}
\end{proposition}

\begin{proof}
We only prove the statement under the assumption $k_1(L)=1$ and $k_1(M)=1$. In the other case the result can be derived from the first one by applying $-^\circ$. If $p_2(L)\neq p_1(M)$ then $L\startilde^\kappa_H M = 0$, by Corollary~\ref{cor aLMN}. From now on we assume that $p_2(L)=p_1(M)$. 

First we show that if $N\le G\times K$ is such that $(L,M)\in\scrY_N$ then $N=L*M$. In fact, since $(L,M)\in\scrY_N$, we have $p_2(N)=p_2(M)$. Moreover, Lemma~\ref{lem butterfly}(i) implies that $p_2(M)=p_2(L*M)$ and $1=k_1(L)=k_1(L*M)$. Thus, $p_2(N)=p_2(L*M)$, and $N\le L*M$ implies $k_1(N)=1$. Equation~(\ref{eqn |L|}) yields
\begin{equation*}
  |N|=|k_1(N)|\cdot|p_2(N)| = |k_1(L*M)|\cdot|p_2(L*M)| = |L*M|,
\end{equation*}
and with $N\le L*M$ we obtain $N=L*M$. Thus, the claim holds and, by Definition~\ref{def aLMN}, we obtain $L\startilde^\kappa_H M = a^{L,M}_{L*M}\cdot L*M$. 

Next we compute $a^{L,M}_{L*M}$ using the formula in Theorem~\ref{thm aLMN}. Assume that $(L',M')\in\scrY_{L*M}$ with $(L',M')\le (L,M)$. Then clearly $L*M=L'*M'$. Moreover, since $p_2(L')=p_1(M')$ and $p_2(L)=p_1(M)$, Lemma~\ref{lem butterfly}(i) implies $p_2(M)=p_2(L*M)=p_2(L'*M')=p_2(M')$. Further, $k_1(M')\le k_1(M)=1$ implies $k_1(M')=1$. Now Equation~(\ref{eqn |L|}) yields $M'=M$, and we obtain $p_2(L')=p_1(M')=p_1(M)=p_2(L)$. Finally, Lemma~\ref{lem butterfly}(i) and $k_1(M)=k_1(M')=1$ imply $k_1(L*M)=k_1(L)$ and $k_1(L'*M')=k_1(L')$. Thus $k_1(L)=k_1(L')$, and Equation~(\ref{eqn |L|}) implies $L'=L$. This shows that $a^{L,M}_{L*M} = \kappa(L,M) = 1/|H|$, and the proof of the proposition is complete.
\end{proof}

\begin{remark}
(a) Proposition~\ref{prop left-free} allows us to recover Theorem~4.7 in \cite{BD11}, where a ghost ring for the left-free double Burnside ring $B^\lhd(G,G)$ was introduced. This is the $\ZZ$-span of the standard basis elements $[(G\times G)/L]$ of $B(G,G)$, with $L\le G\times G$ satisfying $k_1(L)=1$. Warning: in \cite{BD11} the mark homomorphism $\rho_{G,G}$ was defined differently, with an additional scaling factor. 

\smallskip
(b) Note that in \cite{BD11} it was possible to define an integral version for the ghost ring of $B^\lhd(G,G)$, cf.~\cite[Lemma~4.5(c)]{BD11}. In order to generalize this result to our situation one would have to find non-zero rational numbers 
$b_L$, for $L\in\scrS_{G\times G}$, such that
the product
\begin{equation*}
  \bigl(b_L\cdot [L]_{G\times G}^+\bigr)\ \startilde_G^\kappa\ \bigl(b_M\cdot [M]_{G\times G}^+)
\end{equation*}
is contained in the $\ZZ$-span of the elements $b_N\cdot[N]_{G\times G}^+$, $N\in\scrS_{G\times G}$. We do not know whether this is possible.
\end{remark}


\section{Twisted category algebras}\label{sec category algebras}
Throughout this section, $R$ denotes a commutative ring and $\catC$ denotes a 
category whose objects form a set, denoted by $\Ob(\catC)$. We denote by $\Mor(\catC)$ the set of morphisms 
of $\catC$.

\smallskip
The main purpose of this section is to recall basic notions associated with Green's theory of idempotent condensation and with twisted category algebras. We will mostly be
concerned with the case where $\catC$ is a finite category, that is, where $\Mor(\catC)$ (and consequently $\Ob(\catC)$) is a finite set. Category algebras generalize the concept of monoid algebras: see \cite{Gr1}, \cite{St},
\cite{L}, \cite{LS} for the development of these topics. In Sections~\ref{sec simples of A}, \ref{sec simple biset functors} and \ref{sec cyclic} we will apply Green's theory, Theorem~\ref{thm simple} (by Linckelmann and Stolorz, cf.~\cite{LS}) and Theorem~\ref{thm matrix} (by Linckelmann, cf.~\cite{L})
to algebras related to the category of biset functors, which in turn are related to the double Burnside algebra by specializing to categories $\catC$ with just one object.

\smallskip
Two morphisms $s\colon X\to Y$ and $t\colon Y'\to Z$ are called 
{\em composable} if $Y=Y'$. In this case we also say that $t\circ s$ {\em exists}.

\begin{definition}\label{defi cat}
(a)\, The {\em category algebra} $R\catC$ is the $R$-algebra defined as follows: 
the underlying $R$-module is free with basis $\Mor(\catC)$. The product $ts$ of 
two morphisms $s$ and $t$ is defined as $t\circ s$ if $t\circ s$ exists, and it 
is defined to be $0$ if $t\circ s$ does not exist.

\smallskip
(b)\, A {\em $2$-cocycle} of $\catC$ with values in the unit group $R^\times$ is a 
function $\alpha$ that assigns to any two morphisms $s$ and $t$ such that 
$t\circ s$ exists an element $\alpha(t,s)\in R^\times$ with the following property:
for any three morphisms $s,t,u$ of $\catC$ such that $t\circ s$ and $u\circ t$ 
exist, one has
$\alpha(u\circ t,s)\alpha(u,t)=\alpha(u,t\circ s)\alpha(t,s)$.
The {\em twisted category algebra} $R_{\alpha}\catC$ is the free $R$-module with
$R$-basis $\Mor(\catC)$, and with 
multiplication defined by
\begin{equation*}
  t\cdot s:=\begin{cases}
            \alpha(t,s)\cdot t\circ s &\text{ if $t\circ s$ exists,}\\
            0 &\text{ otherwise.}
\end{cases}
\end{equation*}
\end{definition}

\begin{remark}\label{rem infinite cat}
If $\catC$ has finitely many objects then the sum of the identity morphisms is an identity element of $R\catC$, and also $R_\alpha\catC$ has an identity element as we will see. If $\catC$ has infinitely many objects then neither of
the $R$-algebras $R\catC$ and $R_\alpha\catC$ has an identity element. 
\end{remark}

\begin{nothing}\label{noth idemp}{\bf Idempotent morphisms in categories.}\,
We recall several notions from \cite{LS}.

\smallskip
(a)\, The set of idempotent endomorphisms
in $\Mor(\catC)$ carries a poset structure:
if $e,f\in\End_{\catC}(X)$
are idempotent morphisms of the same object $X$ of $\catC$ then we set
\begin{equation*}
  e\leq f:\Longleftrightarrow e=e\circ f=f\circ e\,.
\end{equation*}
If $e$ and $f$ are idempotent morphisms of different objects then they are not 
comparable.

\smallskip
(b)\, For any object $X$ of $\catC$ and any idempotent $e\in \End_{\catC}(X)$,
the group of invertible elements in 
$e\circ\End_{\catC}(X)\circ e$ is denoted by $\Gamma_e$, and we set 
$J_e:=e\circ\End_{\catC}(X)\circ e\smallsetminus \Gamma_e$.

\smallskip
(c)\, Let $X$ and $Y$ be objects of $\catC$, and let $e\in\End_{\catC}(X)$
and $f\in\End_{\catC}(Y)$ be idempotents. We call $e$ and $f$ {\em equivalent}
if there exist morphisms $s\in f\circ\Hom_{\catC}(X,Y)\circ e$
and $t\in  e\circ\Hom_{\catC}(Y,X)\circ f$ such that $t\circ s=e$
and $s\circ t=f$; in this case, the morphisms $s$ and $t$ induce mutually
inverse group isomorphisms
\begin{equation*}
  \Gamma_e\to \Gamma_f,\; u\mapsto s\circ u\circ t;\quad \Gamma_f\to \Gamma_e,\; v\mapsto t\circ v\circ s\,.
\end{equation*}
We remark that in \cite{L}, for instance, the idempotents
$e$ and $f$ are called isomorphic rather than equivalent. We will later 
consider categories whose morphisms are finite groups, and it will then be 
important to not confuse isomorphisms between groups with equivalence
of morphisms in the relevant category, whence the different terminology.

\smallskip
(d)\, Assume now that $\catC$ is finite. Let $X$ and $Y$ be objects of $\catC$ and
let $e\in\End_{\catC}(X)$ and $f\in\End_{\catC}(Y)$ be idempotents. The 2-cocycle 
$\alpha$ restricts to 2-cocycles on the groups $\Gamma_e$ and $\Gamma_f$, and we may consider the 
twisted group algebras $R_{\alpha}\Gamma_e$ and $R_{\alpha}\Gamma_f$ as subrings of $R_{\alpha}\catC$. 
Let $T$ be a simple $R_{\alpha}\Gamma_e$-module and let $T'$ be a simple 
$R_{\alpha}\Gamma_f$-module. The pairs $(e,T)$ and $(f,T')$ are called {\em isomorphic} 
if there exist $s$ and $t$ as in (c) such that $t\circ s=e$ and $s\circ t =f$, 
and such that the isomorphism classes of $T$ and $T'$ correspond to each other 
under the $R$-algebra isomorphism $R_{\alpha}\Gamma_e\cong R_{\alpha}\Gamma_f$ induced by $s$ 
and $t$, cf.~\cite[Proposition~ 5.2]{LS} for a precise definition of this isomorphism. 
For different choices of $s$ and $t$ the different isomorphisms differ only by inner 
automorphisms, cf.~\cite[Proposition~5.4]{LS}. Thus, the definition of $(e,T)$ and 
$(f,T')$ being isomorphic does not depend on the choice of $s$ and $t$.
\end{nothing}

\begin{nothing}\label{noth condense}
{\bf Condensation functors.}\, 
We recall Green's theory of idempotent condensation, cf.~\cite[Section~6.2]{Gr2}.

\smallskip
Suppose that $A$ is any ring with identity and that
$e\in A$ is an idempotent. 
Then $eAe$ is a ring with identity $e$ and one has an exact functor
\begin{equation}\label{eqn condense}
  A\textbf{-Mod}\to eAe\textbf{-Mod},\; V\mapsto eV\,,
\end{equation}
which is often called the {\it condensation functor} with respect to $e$.

One also has a functor 
\begin{equation}\label{eqn uncondense}
  eAe\textbf{-Mod}\to A\textbf{-Mod},\; W\mapsto Ae\otimes_{eAe}W\,;
\end{equation}
this functor is in general not exact.

Whenever $S$ is a simple $A$-module, the $eAe$-module $eS$ is either 0
or again simple, and every simple $eAe$-module occurs in this way.
The construction $S\mapsto eS$ induces a bijection between the isomorphism
classes of simple $A$-modules that are not annihilated by $e$ and the 
isomorphism classes of simple $eAe$-modules. Its inverse can be described as
follows: for every simple $eAe$-module $T$, the $A$-module $M:=Ae\otimes_{eAe}T$
has a unique simple quotient $S$, and $eS\cong T$ as $eAe$-modules. Thus, $S=\Hd(M):=M/\Rad(M)$, the {\em head} of $M$, where $\Rad(M)$ denotes the radical of the $A$-module $M$, i.e., the intersection of all its maximal submodules. 
\end{nothing}

With this notation, one also has the following basic statements.

\begin{proposition}\label{prop schur algebras}
Let $A$ be a unitary ring, let $e\in A$ be an idempotent, and let 
$T$ be a simple $eAe$-module. Moreover, set $M:=Ae\otimes_{eAe}T$ and $S:=M/\Rad(M)=\Hd(M)$. Then:

\smallskip
{\rm (a)}\, $eM\cong T$ as $eAe$-modules;

\smallskip
{\rm (b)}\, $e\Rad(M)=\{0\}$;

\smallskip
{\rm (c)}\, $AeM=M$.
\end{proposition}

\begin{proof}
Since $eS\cong T$ as $eAe$-modules, Parts~(a) and (b) are equivalent. Moreover, Part~(a) is proved in \cite[Section~6.2]{Gr2}.
As for (c), recall that either $AeM=M$ or $AeM\subseteq\Rad(M)$, since
$AeM$ is an $A$-submodule of $M$. If $AeM\subseteq\Rad(M)$
then $\{0\}=eAeM=(eAe)\cdot eM$, by (b), thus $\{0\}=(eAe)\cdot T=T$, by (a), which is impossible, whence (c).
\end{proof}

In Section \ref{sec simple biset functors} we will have to compare 
the annihilator of the simple $A$-module $S$ with the annihilator of the 
simple $eAe$-module $T$. The assertion of the following lemma is probably well known; 
we do, however, include a proof for the reader's convenience.

\begin{lemma}\label{lem Schur ann}
Let $A$ be a unitary ring, let $e\in A$ be an idempotent, let
$T$ be a simple $eAe$-module, and let $S$ be the simple head
of the $A$-module $Ae\otimes_{eAe} T$. Then, for $x\in A$, one has
\begin{equation*}
  xS=\{0\}\ \text{ if and only if }\ (eAxAe)T=\{0\}\,.
\end{equation*}
\end{lemma}

\begin{proof}
Again we set $M:=Ae\otimes_{eAe}T$.
Since $S=M/\Rad(M)$, we infer that $xS=\{0\}$ if and only if $xM\subseteq \Rad(M)$, which in turn
is equivalent to $AxM\subseteq \Rad(M)$, since $\Rad(M)$ is an $A$-submodule
of $M$. Since $AxM$ is an $A$-submodule of $M$, it is either equal to $M$ or is
contained in $\Rad(M)$. Thus, by Proposition \ref{prop schur algebras} (b), $AxM\subseteq\Rad(M)$ if
and only if $eAxM=\{0\}$. By Proposition \ref{prop schur algebras} (c), we have
$eAxM=eAxAeM=(eAxAe)eM$, and by Proposition \ref{prop schur algebras} (a), $eM\cong T$ as
$eAe$-modules. This now implies that $eAxM=\{0\}$ if and only if $(eAxAe)\cdot T=\{0\}$, and the proof of the lemma is complete.
\end{proof}

\begin{nothing}\label{noth condense cat alg}
{\bf Condensation functors and twisted category algebras.}\, 
Now consider the case where $A=R_\alpha\catC$, for
a finite category $\catC$ and a $2$-cocycle $\alpha$
of $\catC$ with coefficients in $R^\times$. Theorem \ref{thm simple}
below is due to Linckelmann--Stolorz and describes how 
the simple $A$-modules can be constructed via the functor in (\ref{eqn uncondense}).
In order to state the theorem, we fix some further notation:
let $e\in\End_{\catC}(X)$ be an idempotent endomorphism in $\catC$. Then
$$e':=\alpha(e,e)^{-1}e$$ 
is an idempotent in $R_\alpha\catC$, and we
also have 
$e'\cdot R_\alpha\catC\cdot e'=R_\alpha(e\circ\End_{\catC}(X)\circ e)$.
One has an $R$-module decomposition
\begin{equation}\label{eqn decomp}
  R_\alpha(e\circ\End_{\catC}(X)\circ e)=R_\alpha\Gamma_e \oplus RJ_e\,,
\end{equation}
where the second summand denotes the $R$-span of $J_e$. Note that $RJ_e$ is an ideal and that
$R_{\alpha}\Gamma_e$ is a unitary $R$-subalgebra of 
$e'\cdot R_\alpha\catC\cdot e'$. For every
$R_\alpha\catC$-module $V$, the condensed module $e'V$ becomes
an $R_{\alpha}\Gamma_e$-module by restriction from $R_\alpha(e\circ\End_{\catC}(X)\circ e)$ to 
$R_\alpha \Gamma_e$. Conversely, given an $R_{\alpha}\Gamma_e$-module $W$,
we obtain the $R_\alpha\catC$-module
$R_\alpha\catC e'\otimes_{e' R_\alpha\catC e'}\Wtilde$, where $\Wtilde$ denotes the inflation of $W$ from $R_\alpha\Gamma_e$ to $e'R_\alpha\catC e'$ with respect to the decomposition~(\ref{eqn decomp}).
\end{nothing}

With this notation one has the following theorem, due to Linckelmann and Stolorz.

\begin{theorem}[\cite{LS}, Theorem 1.2]\label{thm simple}
Let $\catC$ be a finite category, and let $\alpha$ be a $2$-cocycle of $\catC$ with coefficients
in $R^\times$. There is a bijection between the set of isomorphism
classes of simple $R_\alpha\catC$-modules and the set of
isomorphism classes of pairs $(e,T)$, where $e$ is an idempotent
endomorphism in $\catC$ and $T$ is a simple $R_{\alpha}\Gamma_e$-module;
the isomorphism class of a simple $R_\alpha\catC$-module $S$ is
sent to the isomorphism class of the pair $(e,e'S)$, where
$e$ is an idempotent endomorphism that is minimal
with the property $e'S\neq \{0\}$. Conversely, the
isomorphism class of a pair $(e,T)$ is sent to the isomorphism class
of the $R_\alpha\catC$-module 
$\Hd(R_\alpha\catC e'\otimes_{e' R_\alpha\catC e' }\Ttilde)$.
\end{theorem}

\begin{definition}[\cite{Ka}]\label{defi inverse}
The category $\catC$ is called
an {\it inverse category} if, for every $X,Y\in\Ob(\catC)$ and
every $s\in\Hom_{\catC}(X,Y)$, there is a unique $\shat\in\Hom_{\catC}(Y,X)$
such that
\begin{equation*}
  s\circ\shat\circ s=s\quad \text{ and }\quad \shat\circ s \circ \shat = \shat\,.
\end{equation*}
\end{definition}

\begin{nothing}\label{noth inverse cat}
{\bf Category algebras of inverse categories.}\,
From now on, let $\catC$ be a finite inverse category. Note that if $s$ is a morphism in 
$\catC$ then $\hat{\shat} = s$ and the morphisms $\shat\circ s$ and $s\circ\shat$ are idempotent morphisms.
We recall several constructions from \cite{L}, which are based on constructions in~\cite{St} for inverse monoids.

\smallskip
(a)\, One can extend the poset structure on the set of idempotent morphisms of $\Mor(\catC)$ from~\ref{noth idemp}(a) to the set $\Mor(\catC)$ of all morphisms: for
$X,Y\in\Ob(\catC)$ and $s,t\in\Hom_{\catC}(X,Y)$, one defines
\begin{equation*}
  s\leq t:\Longleftrightarrow s=t\circ e,\text{ for some idempotent } e\in \End_{\catC}(X)\,.
\end{equation*}
If $s$ and $t$ are morphisms with different domain or codomain then they are incomparable.
For $s\in\Mor(\catC)$ we consider the element
\begin{equation*}
  \underline{s}:= \sum_{t\le s} \mu_{t,s} t\in R\catC
\end{equation*}
where $\mu_{t,s}$ denotes the M\"obius function with respect to the partial order on $\Mor(\catC)$. 
Note that the elements in
$\{\underline{s}\mid s\in\Mor(\catC)\}$ form an $R$-basis of $R\catC$. 

\smallskip
(b)\, Let $X$ and $Y$ be objects of $\catC$ and let $e\in\End_{\catC}(X)$ and $f\in\End_{\catC}(Y)$ be idempotent endomorphisms.
Then \ref{noth idemp}(c) and Definition \ref{defi inverse}
show that $e$ and $f$ are equivalent if and only if there exists some
$s\in f\circ\Hom_{\catC}(X,Y)\circ e$ such that
$\shat\circ s=e$ and $s\circ \shat = f$. 
Moreover, the group $\Gamma_e$ then consists of those
endomorphisms $u\in\End_{\catC}(X)$ with
$\hat{u}\circ u=e=u\circ\hat{u}$.

\smallskip
(c)\, From the inverse category $\catC$ one can construct another finite category $G(\catC)$ 
as follows:

\smallskip
\quad (i)\, the objects in $G(\catC)$ are the pairs $(X,e)$, where
$X\in\Ob(\catC)$ and $e$ is an idempotent in $\End_{\catC}(X)$;

\quad (ii)\, the morphisms in $G(\catC)$ are triples $(f,s,e):(X,e)\to (Y,f)$,
where $s\in f\circ\Hom_{\catC}(X,Y)\circ e$ is such that $\shat\circ s=e$ and $s\circ\shat=f$; 
in particular, $e$ and $f$ are equivalent endomorphisms in $\catC$; 

\quad (iii)\, if $(f,s,e):(X,e)\to (Y,f)$ and $(g,t,f):(Y,f)\to (Z,g)$
are morphisms in $G(\catC)$ then their composition 
$(X,e)\to (Z,g)$
is defined as $(g,t,f)\circ (f,s,e):=(g,t\circ f\circ s,e)$.

\smallskip
Note that $G(\catC)$ is a {\em groupoid}, i.e., a category in which every morphism is an
isomorphism.
\end{nothing}

\begin{lemma}[\cite{L}, Theorem 4.1; \cite{St}, Theorem 4.2]\label{lemma cat iso}
Let $\catC$ be a finite inverse category. Then the category algebras $R\catC$ and $RG(\catC)$ are
isomorphic. More precisely, the maps
\begin{equation*}
RG(\catC)\to R\catC\,,\ (f,s,e)\mapsto \underline{s}\,,\quad \text{and} \quad
R\catC\to RG(\catC)\,,\ s\mapsto \sum_{t\leq s}(t\circ\that,t,\that\circ t)\,,
\end{equation*}
define mutually inverse $R$-algebra isomorphisms. 
\end{lemma}

\begin{remark}\label{rem groupoid algebra}
It is well known that the category algebra $R\catD$ of a finite groupoid $\catD$ 
is isomorphic to a direct product of matrix algebras over group algebras. More precisely, 
let $E$ be a set of representatives of the isomorphism classes of objects $e$ of $\catD$, 
and for each $e\in E$ let $n(e)$ denote the number of objects isomorphic to $e$.
Since $\catD$ is a groupoid, every endomorphism of $e\in E$ is an automorphism of $e$,
and $\id_e$ is the unique idempotent in $\End_\catD(e)$; in particular, we have
$\Gamma_{\id_e}=\End_\catD(e)$.
For convenience we may thus denote $\End_\catD(e)$ by $\Gamma_e$.

Note that, in the case where $\catD=G(\catC)$ for some finite inverse category
$\catC$, this is consistent with \ref{noth idemp}(b): if $(X,e)\in G(\catC)$ then
$\End_{G(\catC)}((X,e))=\{(e,s,e)\mid s\in e\circ\End_\catC(X)\circ e,\ \shat\circ s=e=s\circ\shat\}$,
and $\{s\in e\circ\End_\catC(X)\circ e\mid \shat \circ s=e= s\circ\shat \}$ is the set of
invertible elements in $e\circ \End_C(X)\circ e$, thus equal to $\Gamma_e$ as defined 
in \ref{noth idemp}(b).

Suppose again that $\catD$ is an arbitrary finite groupoid.
Then there exists an $R$-algebra isomorphism
\begin{equation}\label{eqn groupoid iso}
  \epsilon\colon R\catD\to \bigtimes_{e\in E} \Mat_{n(e)} (R\Gamma_e),
\end{equation}
which can be defined as follows: we fix $e\in E$, denote by $e=e_1,\ldots,e_{n(e)}$ the objects that lie in the isomorphism class of $e$, and choose an 
isomorphism $s_i\colon e\to e_i$ for every $i\in\{1,\ldots,n(e)\}$. Let $s$ be a morphism in $\catD$. 
If the domain of $s$ is not isomorphic to $e$ then the matrix in the $e$-component of $\epsilon(s)$ 
is defined to be $0$. If the domain of $s$ is isomorphic to $e$ then there exist unique elements 
$i,j\in\{1,\ldots,n(e)\}$ such that $s\colon e_j\to e_i$. In this case the $e$-component of $\epsilon(s)$ 
is the matrix all of whose entries are equal to $0$ except for the $(i,j)$-entry, which is equal to 
$s_i^{-1}\circ s\circ s_j$. The isomorphism $\epsilon$ depends on the ordering of the elements in 
an isomorphism class and on the choices of the isomorphisms $s_i$. 
However, any two isomorphisms defined as above differ only by an inner automorphism.
\end{remark}

Combining the two isomorphisms in Lemma~\ref{lemma cat iso} and 
Remark~\ref{rem groupoid algebra} one obtains the following theorem, due to Linckelmann, 
cf.~\cite[Corollary 4.2]{L}.

\begin{theorem}\label{thm matrix}
Let $\catC$ be a finite inverse category. Let $E$ be a set of representatives of the equivalence
classes of idempotents in $\Mor(\catC)$ and, for $e\in E$, let $n(e)$ denote the cardinality of the equivalence class of $e$. 
Then there exists an $R$-algebra isomorphism
\begin{equation}\label{eqn omega}
  \omega\colon R\catC\to \bigtimes_{e\in E} \Mat_{n(e)} (R\Gamma_e)\,.
\end{equation}
\end{theorem}

\begin{remark}\label{rem omega}
(a) Using the explicit description of the isomorphism $\epsilon$ in 
Remark~\ref{rem groupoid algebra} we obtain the following description of the $e$-component 
of the isomorphism $\omega$ in Theorem~\ref{thm matrix}: let $e=e_1,\ldots,e_{n(e)}$ denote the 
idempotent morphisms of $\catC$ that are equivalent to $e$ and for every $i\in\{1,\ldots, n(e)\}$ let $s_i$ 
be a morphism such that $\shat_i\circ s_i = e$ and $s_i\circ\shat_i=e_i$. For every morphism 
$s$ of $\catC$, the $e$-component of $\omega(\underline{s})$ is defined to be $0$ if the idempotent 
$\shat\circ s$ is not equivalent to $e$, and if it is equivalent to $e$ then there exist unique 
elements $i,j\in\{1,\ldots,n(e)\}$ such that $\shat\circ s = e_j$ and $s\circ\shat= e_i$. In 
this case the $e$-component of $\omega(\underline{s})$ is the matrix all of whose entries are equal to $0$ except for
the $(i,j)$-entry, which is equal to $s_i \circ s\circ \shat_j$. 

\smallskip
(b) It follows immediately from (a) that the central idempotent of $R\catC$ corresponding to the 
identity matrix in the $e$-component, for $e\in E$ and with the notation introduced in (a), is equal to 
\begin{equation*}
 \tilde{e}:= \sum_{i=1}^{n(e)} \underline{e_i}\in R\catC\,.
\end{equation*}
This element is independent of the choices made when constructing the isomorphism $\omega$. Moreover,
 the $e$-th component $\Mat_{n(e)}(R\Gamma_e)$ in (\ref{eqn omega}) corresponds to a two-sided ideal $I_e= R\catC\cdot \tilde{e}$, which is independent of the choices involved in the definition of $\omega$, and one has $R\catC=\bigoplus_{e\in E} I_e$. By (a), the elements $\underline{s}\in R\catC$, where $s$ runs through all morphisms in $\catC$ such that the idempotent $\shat\circ s$ is equivalent to $e$, form an $R$-basis of $I_e$.
\end{remark}

We will use Theorems \ref{thm simple} and \ref{thm matrix} on category algebras in two situations.
The following example describes the first one; it will be used in Sections~\ref{sec simples of A} and \ref{sec simple biset functors}.

\begin{example}\label{ex Bouc algebra}
(a) We denote by $\catB$ a category with the following properties: its objects form a set of finite groups 
such that every isomorphism type of finite group is represented
in $\Ob(\catB)$. The morphisms in $\catB$ are defined by
\begin{equation*}
  \Hom_{\catB}(H,G):=\scrS_{G\times H}\,,
\end{equation*}
for $G,H\in\Ob(\catB)$,
and the composition is defined by
\begin{equation*}
  L\circ M:=L*M\,,
\end{equation*}
for $L\in \scrS_{G\times H}$, $M\in\scrS_{H\times K}$, with $G,H,K\in\Ob(\catB)$. 
The identity element of the object $G$ of $\catB$ is equal to $\Delta(G)$. 

\smallskip
(b) In the sequel we assume that $\catC$ is a subcategory of $\catB$. For $G,H\in\Ob(\catC)$, 
we set $\catC_{G,H}:=\Hom_{\catC}(H,G)$. For completeness we also set $\catC_{G,H}:=\emptyset$ if 
$G\in\Ob(\catB)\smallsetminus\Ob(\catC)$ or $H\in\Ob(\catB)\smallsetminus\Ob(\catC)$. Let $R$ be a commutative ring such that $|G|\in R^\times$ 
for all $G\in \Ob(\catC)$. Then, by Proposition~\ref{prop cocycle}, the function $\kappa$ defined by
\begin{equation*}
  \kappa(L,M):=\frac{|k_2(L)\cap k_1(M)|}{|H|}\in R^\times\,,
\end{equation*}
with $L\in\catC_{G,H}$, $M\in\catC_{H,K}$, for $G,H,K\in\Ob(\catC)$, defines a $2$-cocycle on $\catC$.
We obtain a twisted category 
algebra $R_\kappa\catC$, which we will denote by $A_{\catC,R}^\kappa$.

\smallskip
(c) Recall from \cite[Page 105]{W} that biset functors over $R$, for an arbitrary commutative ring $R$, can be 
considered as modules for the algebra $\bigoplus_{G,H\in\Ob(\catB)} RB(G,H)$, where the multiplication on 
two components $RB(G,H)$ and $RB(H',K)$ is induced by the tensor product of bisets over $H$ if $H=H'$,
and is defined to be 0 if $H\neq H'$. Note that this algebra has no identity element, since $\Ob(\catB)$
is infinite.
If $\catC$ is a subcategory of $\catB$ and if $\catC_{G,H}\subseteq \scrS_{G\times H}$ is closed under $G\times H$-conjugation then we 
define
\begin{equation*}
  RB^\catC(G,H):=\langle [(G\times H)/L] \mid L\in \catC_{G,H}\rangle_R\subseteq RB(G,H)\,,
\end{equation*}
the $R$-span of the elements $[G\times H/L]$ with $L\in\catC_{G,H}$.
If $\catC$ is finite then the corresponding $R$-subalgebra $\bigoplus_{G,H\in\Ob(\catC)} RB^\catC(G,H)$ of $\bigoplus_{G,H\in\Ob(\catB)} RB(G,H)$
has an identity element. If, moreover, $|G|\in R^\times$ for every $G\in\Ob(\catC)$ then the maps $\alpha_{G,H}$, $G,H\in\Ob(\catC)$, yield an $R$-algebra isomorphism 
\begin{equation}\label{eqn Bouc algebra}
     \bigoplus_{G,H\in\Ob(\catC)} RB^\catC(G,H) \cong e_\catC  A_{\catC,R}^\kappa  e_\catC\,,
\end{equation}
 where
 \begin{equation*}
   e_\catC:=\sum_{G\in\Ob(\catC)} e_G
 \end{equation*}
is an idempotent in $A_{\catC,R}^\kappa$, cf.~Proposition~\ref{prop fixed points} and the paragraph preceding it. Recall that $e_G:=\sum_{g\in G}\Delta_g(G)$, for $G\in\Ob(\catC)$.
Many considerations about biset functors on the 
category of all finite groups can be reduced to the case of finitely many groups. Thus, 
the algebra $A_{\catC,R}^\kappa$ through its condensed algebra (\ref{eqn Bouc algebra}) can be used to study
biset functors over $R$.
 \end{example}
 
\bigskip 
We will study the simple modules of $A_{\catC,R}^\kappa$ and 
$e_\catC  A_{\catC,R}^\kappa  e_\catC$ in Sections~\ref{sec simples of A} and \ref{sec simple biset functors}, respectively, for appropriate choices of $\catC$ and $R$.
 
We will see in Section~\ref{sec cyclic} that, in special cases, the multiplication $\startilde^\kappa_H$ from Definition~\ref{def aLMN} can be related to the following construction.

\bigskip
For the remainder of this section we fix a set $\scrD$ of finite
groups. We will introduce an inverse category $\catCtilde=\catCtilde(\scrD)$ related to
this set $\scrD$. Later, in Section \ref{sec cyclic}, we will
see that in the case where the groups in $\scrD$ are cyclic
we obtain an $R$-algebra isomorphism between the category algebra
$R\catCtilde$ and $\bigoplus_{G,H\in\scrD}RB(G,H)$. There, $R$ will
denote a commutative ring such that $|G|$ and $|\Aut(G)|$ are
units in $R$, for all $G\in\scrD$.

\begin{definition}\label{defi group cat}
For finite groups $G$ and $H$ and a subgroup $L\le G\times H$, we set
\begin{equation*}
  \scrP(L):=\{L'\le L\mid p_1(L')=p_1(L)\text{ and } p_2(L')=p_2(L)\}\,.
\end{equation*}

Given the set $\scrD$ of finite groups,
we define a category $\catCtilde=\catCtilde(\scrD)$ as follows:

(i)\, $\Ob(\catCtilde)=\{(G,G')\mid G\in\scrD, G'\le G\}$;

(ii)\, for $(G,G'),(H,H')\in\Ob(\catCtilde)$, we set 
$\Hom_{\catCtilde}((H,H'),(G,G')):=\scrP(G'\times H')$;

(iii)\, for $(G,G'),(H,H'), (K,K')\in\Ob(\catCtilde)$, $\GLH\in\Hom_{\catCtilde}((H,H'),(G,G'))$,
and $\HMK\in \Hom_{\catCtilde}((K,K'),(H,H'))$, we define the composition
$\GLH\circ \HMK:=\GLH*\HMK.$

Here, we use the notation $\GLH$, since every morphism needs to determine its source and target objects. While $G'=p_1(L)$ and $H'=p_2(L)$ are determined by the notation $L$, $G$ and $H$ are not.
\end{definition}

With this notation, we have the following proposition.

\begin{proposition}\label{prop group cat}
{\rm (a)}\, Let $(G,G'), (H,H')\in\Ob(\catCtilde)$, and
let $\GLH\in\Hom_{\catCtilde}((H,H'),(G,G'))$. Then
$(\GLH)^\circ$ is the unique element $\HMG\in\Hom_{\catCtilde}((G,G'),(H,H'))$ such that

\begin{equation*}
  \GLH*\HMG *\GLH=\GLH \quad\text{and}\quad \HMG *\GLH*\HMG=\HMG\,; 
\end{equation*}
in particular, $\catCtilde$ is an inverse category.

\smallskip
{\rm (b)}\, The idempotent endomorphisms in $\catCtilde$ are precisely the endomorphisms of the form
\begin{equation*}
  {}_G(P,K,\id,P,K)_G\qquad (K\unlhd P\leq G\in\scrD)\,.
\end{equation*}
Moreover, two idempotent endomorphisms $e:={}_G(P,K,\id,P,K)_G$ and 
$f:={}_H(P',K',\id,P',K')_H$ in $\catCtilde$ are equivalent if and only if $P/K\cong P'/K'$.

\smallskip
{\rm (c)}\, If $e:={}_G(P,K,\id,P,K)_G$ is an idempotent endomorphism in $\catCtilde$ then
one has $\Gamma_{e}=\{{}_G(P,K,\alpha,P,K)_G\mid \alpha\in\Aut(P/K)\}$; in
particular, $\Gamma_{e}\cong \Aut(P/K)$, by Lemma~\ref{lem butterfly}(ii).
\end{proposition}

\begin{proof}
(a) It follows immediately from Lemma~\ref{lem butterfly} that $\HMG:=(\GLH)^\circ$ has the desired property. Conversely, suppose that $\HMG\in\Hom_{\catCtilde}((G,G'),(H,H'))$ satisfies
$L*M*L=L$ and $M*L*M=M$. By Lemma \ref{lem butterfly}(i), this forces
$k_2(L)=k_1(M)$ as well as $k_1(L)=k_2(M)$. Note that $p_2(L)=H'=p_1(M)$ and $p_1(L)=G'=p_2(M)$ by definition. Now, $L*M*L=L$ and Lemma~\ref{lem butterfly}(i) again force that $\eta_L\circ\eta_M\circ\eta_L=\eta_L$. This implies $\eta_L=\eta_M^{-1}$ and $M=L^\circ$. 

\smallskip
Assertions (b) and (c) are now easy consequences of (a) and \ref{noth inverse cat}.
\end{proof}

\begin{remark}\label{rem group cat}
Proposition \ref{prop group cat} thus shows, in particular, that
the equivalence classes of idempotent endomorphisms in $\catCtilde$ are
in bijection with the isomorphism classes of groups
occurring as subquotients of groups in $\scrD$.
\end{remark}


\section{The simple $A_{\catC,R}^\kappa$-modules}\label{sec simples of A}

Throughout this section, let $\catB$ be a category as in Example~\ref{ex Bouc algebra}(a) and 
let $\catC$ be a finite subcategory of $\catB$. Moreover, let $R$ be a commutative ring such that $|G|$ 
is invertible in $R$ for each $G\in\Ob(\catC)$. In this section we will use Theorem~\ref{thm simple} 
to construct the simple modules of the twisted category $R$-algebra $A_{\catC,R}^\kappa$ introduced 
in \ref{ex Bouc algebra}(b). This is in preparation for Section~\ref{sec simple biset functors}, 
where we determine the simple modules of the $R$-algebra $e_\catC A_{\catC,R}^\kappa e_\catC$, which is 
related to the category of biset functors.
Note that all the results apply, in particular, to the case where the category $\catC$ 
has only one object $G$ and $\catC_{G,G}=\scrS_{G\times G}$.

\smallskip
We say that a finite group $Q$ is {\em realized} ({\em by a morphism}) 
in $\catC$ if there 
exist $G,H\in\Ob(\catC)$ and $L\in\catC_{G,H}=\Hom_{\catC}(H,G)$ such that $q(L)=[Q]$, the isomorphism class 
of $Q$; recall that $q(L)$ denotes the isomorphism class of $p_1(L)/k_1(L)$. 
Note that, in particular, each object $G$ of $\catC$ is realized in $\catC$
by the identity morphism $\Delta(G)\in\Hom_{\catC}(G,G)$.

Throughout this section, we 
will assume that $\catC$ satisfies the following property:
\begin{equation}\label{eqn closed}
\parbox{12cm}{
  If $Q$ is a finite group that is realized in $\catC$, if $G$ and $H$ are objects in $\catC$, and 
if $(P_1,K_1)$ and $(P_2,K_2)$ are sections of $G$ and $H$, respectively, such that $P_1/K_1\cong Q\cong P_2/K_2$, 
then there exists an isomorphism $\alpha\colon P_2/K_2\myiso P_1/K_1$ such that $(P_1,K_1,\alpha,P_2,K_2)\in \catC_{G,H}$.}
\end{equation}

\smallskip
We first derive two consequences from Condition~(\ref{eqn closed}). For a section $(P,K)$ of a group $G$ we set 
\begin{equation*}
  e_{(P,K)}:=(P,K,\id_{P/K},P,K)\in\scrS_{G\times G}\,;
\end{equation*}
note that, by Lemma~\ref{lem butterfly}(ii), the element $e_{P,K}$ is an idempotent in the monoid $(\scrS_{G\times G},*)$.

\begin{lemma}\label{lem closed 1}
Let $Q$ be a finite group that is realized by a morphism in $\catC$. 
Let $G$ be an object of $\catC$ and let $(P,K)$ be a section of $G$ such that $P/K\cong Q$. 
Then the idempotent $e_{(P,K)}$ belongs to $\catC_{G,G}$.
\end{lemma}

\begin{proof}
Since $\catC$ satisfies (\ref{eqn closed}), there exists an automorphism $\alpha\in\Aut(P/K)$ such 
that $(P,K,\alpha,P,K)\in C_{G,G}$. By Lemma~\ref{lem butterfly}(ii), the $|\Aut(P/K)|$-th power of $(P,K,\alpha,P,K)$ 
is equal to $e_{(P,K)}$, showing that $e_{(P,K)}$ belongs to $\catC_{G,G}$.
\end{proof}

For a section $(P,K)$ of an object $G$ of $\catC$ with $e_{(P,K)}\in\catC_{G,G}$, we set
\begin{equation*}
  \Aut_{\catC}(P/K):=\{\alpha\in\Aut(P/K)\mid (P,K,\alpha,P,K)\in\catC_{G,G}\}\,.
\end{equation*}
Clearly, this is a subgroup of $\Aut(P/K)$. The following lemma shows that the isomorphism type of 
$\Aut_{\catC}(P/K)$ does in fact only depend on the isomorphism type of the group $P/K$ (not 
on the actual section $(P,K)$). We omit the straightforward proof.

\begin{lemma}\label{lem closed 2}
Assume that $G$ and $G'$ are objects of $\catC$ and that $(P,K)$ and $(P',K')$ are sections 
of $G$ and $G'$, respectively, such that $e_{(P,K)}\in\catC_{G,G}$, $e_{(P',K')}\in\catC_{G',G'}$ and 
$P/K\cong P'/K'$. Let $\gamma\colon P'/K'\myiso P/K$ be an isomorphism such that 
$(P,K,\gamma,P',K')\in\catC_{G,G'}$ (which exists by (\ref{eqn closed})). Then the map
\begin{equation*}
  c_{\gamma}\colon \Aut_{\catC}(P'/K') \to \Aut_{\catC}(P/K)\,,\quad \alpha\mapsto \gamma\alpha\gamma^{-1}\,,
\end{equation*}
is a group isomorphism, and if also $\delta\colon P'/K'\myiso P/K$ is an isomorphism with 
$(P,K,\delta,P',K')\in\catC_{G,G'}$ then $c_\delta=c_{\alpha}\circ c_{\gamma}$, where 
$\alpha:=\delta\circ\gamma^{-1}\in\Aut_{\catC}(P/K)$.
\end{lemma}

Recall from \ref{noth idemp}(c) that two idempotents $e\in\catC_{G,G}$ and $f\in\catC_{H,H}$ 
are called equivalent if there exist elements $s\in f*\catC_{H,G}*e$ and $t\in e*\catC_{G,H}*f$ 
such that $t*s = e$ and $s*t=f$.

\begin{proposition}\label{prop idempotent subgroups}
Let $G$ and $H$ be objects in $\catC$.

\smallskip
{\rm (a)} A morphism $L=(P_1,K_1,\alpha,P_2,K_2)\in\catC_{G,G}$ is an idempotent if and only if 
$(P_1,K_1)$ and $(P_2,K_2)$ are linked and $\alpha = \beta(P_1,K_1;P_2,K_2)$.

\smallskip
{\rm (b)} If $L=(P_1,K_1,\alpha,P_2,K_2)\in\catC_{G,G}$ is an idempotent morphism then $L$ 
is equivalent to $e_{(P_i,K_i)}$, for $i=1,2$.

\smallskip
{\rm (c)} Two idempotent morphisms $L\in\catC_{G,G}$ and $M\in\catC_{H,H}$ 
are equivalent if and only if $q(L)=q(M)$.
\end{proposition}

\begin{proof}
(a) By Lemma~\ref{lem butterfly}, we have $L*L=L$ if and only if $(P_2, K_2)$ and 
$(P_1,K_1)$ are linked and $\alpha = \alpha \circ \beta(P_2,K_2;P_1,K_1) \circ \alpha$. 
But the last equality is equivalent to the equality $\beta(P_2,K_2;P_1,K_1) \circ \alpha = \id_{P_2/K_2}$. 
Since $\beta(P_2,K_2;P_1,K_1)^{-1} = \beta(P_1,K_1;P_2,K_2)$, the result follows.

\smallskip
(b) Suppose that $L*L=L$. 
By Part~(a), the sections $(P_1,K_1)$ and $(P_2,K_2)$ are linked and $\alpha=\beta(P_1,K_1;P_2,K_2)$. 
By Lemma~\ref{lem closed 1}, the morphisms $M_i:=e_{(P_i,K_i)}$ belong to $\catC_{G,G}$, for $i=1,2$. 
Moreover, Lemma~\ref{lem butterfly} implies that 
\begin{equation*}
  L=M_1*L\,,\quad M_1=L*M_1\,,\quad L=L*M_2\,,\quad\text{and}\quad M_2=M_2*L\,.
\end{equation*}
Thus, $L$ and $M_i$ are equivalent for $i=1,2$.

\smallskip
(c) If $L$ and $M$ are equivalent then there exist $S\in\catC_{H,G}$ and $T\in\catC_{G,H}$ such 
that $L=T*S$ and $M=S*T$. This implies $L = L*L = T*S*T*S=T*M*S$ and $M=M*M = S*T*S*T = S*L*T$. 
Therefore, $p_1(L)/k_1(L)$ is isomorphic to a subquotient of $p_1(M)/k_1(M)$ and also $p_1(M)/k_1(M)$ is isomorphic to a subquotient of $p_1(L)/k_1(L)$. This implies $q(L)=q(M)$. Conversely, assume that $q(L)=q(M)$. By Lemma~\ref{lem closed 1} and Part~(b), we may assume that 
$L=e_{(P_1,K_1)}\in\catC_{G,G}$ and $M=e_{(P_2,K_2)}\in\catC_{H,H}$ for a section $(P_1,K_1)$ of $G$ and a 
section $(P_2,K_2)$ of $H$, with $P_1/K_1\cong P_2/K_2$. By (\ref{eqn closed}), there exist isomorphisms 
$\alpha\colon P_1/K_1\myiso P_2/K_2$ and $\beta\colon P_2/K_2\myiso P_1/K_1$ such that 
$S:=(P_2,K_2,\alpha,P_1,K_1)\in\catC_{H,G}$ and $T:=(P_1,K_1,\beta,P_2,K_2)\in\catC_{G,H}$. For 
$T':=T*(S*T)^{|\Aut(P_2/K_2)|-1}\in\catC_{G,H}$ we then obtain $S*T'=e_{(P_2,K_2)}$ and $T'*S=e_{(P_1,K_1)}$, 
and the proof is complete.
\end{proof}

Recall from Section~\ref{sec category algebras} that, for an idempotent morphism $e$ 
in $\catC_{G, G}$, we denote by $\Gamma_e$ the group of invertible elements of the monoid 
$e*\catC_{G, G}*e$. The identity element of $\Gamma_e$ is equal to $e$.

\begin{proposition}\label{prop Gamma_e}
Let $G$ be an object of $\catC$, and let $(P,K)$ be a section of $G$ such that $e:=e_{(P,K)}$ belongs to $\catC_{G,G}$.

\smallskip
{\rm (a)} One has
\begin{equation*}
  \Gamma_e = \{(P,K,\alpha,P,K)\mid \alpha\in \Aut_{\catC}(P/K)\},
\end{equation*}
and the map $\alpha\mapsto (P,K,\alpha,P,K)$ defines a group isomorphism $\Aut_{\catC}(P/K)\to\Gamma_e$.

\smallskip
{\rm (b)} The $2$-cocycle $\kappa$ on the monoid $(\catC_{G,G},*)$ restricted to 
the subgroup $\Gamma_e$ is the constant function with value $[G:K]^{-1}$, and the map 
$\alpha\mapsto [G:K]\cdot(P,K,\alpha,P,K)$ defines an $R$-algebra isomorphism between the 
group algebra $R\Aut_{\catC}(P/K)$ and the twisted group algebra $R_\kappa \Gamma_e$.
\end{proposition}

\begin{proof}
(a) Let $\alpha,\beta\in\Aut_{\catC}(P/K)$. Then $L_\alpha:=(P,K,\alpha,P,K)\in\catC_{G,G}$ 
satisfies $L_\alpha= e*L_\alpha*e$ and $L_\alpha*L_\beta = L_{\alpha\circ\beta}$, by Lemma~\ref{lem butterfly}. 
Thus, $L_\alpha\in\Gamma_e$ for all $\alpha\in\Aut_{\catC}(P/K)$ and we obtain an injective group 
homomorphism $\Aut_\catC(P/K)\to\Gamma_e$. Conversely, for $L\in\Gamma_e$ we have $L=e*L*e$ and there exists 
$M=e*M*e\in\catC_{G, G}$ with $M*L=L*M=e$. Since $L=e*L*e$, we have $p_i(L)\le p_i(e)=P$ and 
$k_i(L)\ge k_i(e) = K$ for $i=1,2$. Since $L*M=e$, we also have $P=p_1(e)\le p_1(L)$ and $K=k_1(e)\ge k_1(L)$. 
Similarly, $M*L=e$ implies $P\le p_2(L)$ and $K\ge k_2(L)$. Altogether we obtain $p_i(L)=P$ and $k_i(L)=K$ 
for $i=1,2$. This implies that $L=(P,K,\alpha,P,K)$ for some $\alpha\in\Aut_{\catC}(P/K)$.

\smallskip
(b) Writing again $L_\alpha:=(P,K,\alpha,P,K)$ for $\alpha\in\Aut_{\catC}(P/K)$, we have 
$\kappa(L_\alpha,L_\beta)= |K\cap K|/|G|=[G:K]^{-1}$, for
$\alpha,\beta\in\Aut_\catC(P/K)$. The remaining statement is an easy verification.
\end{proof}

We call a section $(P_1,K_1)$ of a group $G$ and a section $(P_2,K_2)$ of a group $G'$ 
{\em isomorphic} if $P_1/K_1\cong P_2/K_2$. In this case we write $(P_1,K_1)\cong (P_2,K_2)$. 
For an object $G$ of $\catC$ and a section $(P,K)$ of $G$ with $e_{(P,K)}\in\catC_{G,G}$ we set 
$e'_{(P,K)}:=[G:K]\cdot e_{(P,K)}$, which is the corresponding idempotent in $R_{\kappa}\catC_{G,G}$, 
cf.~\ref{noth condense cat alg}. It is  also the identity in $R_{\kappa}\Gamma_{e_{(P,K)}}$. 
As in \ref{noth idemp}(b) we write $J_{e_{(P,K)}}= (e_{(P,K)} * \catC_{G, G} * e_{(P,K)})\smallsetminus \Gamma_{e_{(P,K)}}$. 
Let $A:=A_{\catC,R}^\kappa$. Recall from \ref{noth condense cat alg} that the $R$-span $RJ_{e_{(P,K)}}$ is an 
ideal of the $R$-algebra $e'_{(P,K)}Ae'_{(P,K)}=e'_{(P,K)}R_\kappa\catC_{G,G}e'_{(P,K)}$, and that we have 
\begin{equation}\label{eqn e' decomposition}
  e'_{(P,K)}A e'_{(P,K)} = R_\kappa\Gamma_{e_{(P,K)}} \oplus RJ_{e_{(P,K)}}\,.
\end{equation}

The following theorem is now an immediate consequence of Theorem~\ref{thm simple} and 
Propositions~\ref{prop idempotent subgroups} and \ref{prop Gamma_e}.

\begin{theorem}\label{thm simples of A}
Let $\calS$ be a set of representatives of the isomorphism classes of sections $(P,K)$ of 
groups $G\in \Ob(\catC)$ with $e_{P,K}\in\catC_{G,G}$, and for each $(P,K)\in\calS$ let $\calT_{(P,K)}$ 
denote a set of representatives of the isomorphism classes of simple left 
$R[\Aut_{\catC}(P/K)]$-modules. Moreover, set $A:=A_{\catC,R}^\kappa$. Then the map
\begin{equation*}
  ((P,K),T) \mapsto 
  \Hd \bigl(A e'_{(P,K)} \otimes_{e'_{(P,K)}Ae'_{(P,K)}} 
	  \Inf_{R_\kappa\Gamma_{e_{(P,K)}}}^{e'_{(P,K)}Ae'_{(P,K)}}(T)\bigr)
\end{equation*}
induces a bijection between the set of pairs $((P,K),T)$, where $(P,K)\in\calS$ and 
$T\in\calT_{(P,K)}$, and the set of isomorphism classes of simple left $A$-modules. Here, the simple 
$R[\Aut_{\catC}(P/K)]$-module $T$ is viewed as an $R_\kappa\Gamma_{e_{(P,K)}}$-module via the isomorphism 
from Proposition~\ref{prop Gamma_e}(b) and is inflated via the decomposition in (\ref{eqn e' decomposition}).
\end{theorem}


\section{Simple $e_\catC A_{\catC,R}^\kappa e_\catC$-modules in
characteristic $0$}\label{sec simple biset functors}

Throughout this section we assume that $R$ is a field of characteristic $0$, that 
$\catB$ is a category as in Example~\ref{ex Bouc algebra}(a) and that $\catC$ is a 
subcategory of $\catB$ satisfying the condition in~(\ref{eqn closed}). 
We also assume throughout 
this section that, for any two objects $G$ and $H$ of $\catC$, the set $\catC_{G,H}$ 
is closed under $G\times H$-conjugation. However, we do not assume at this point that 
$\catC$ is finite, but will do so in some parts of this section. Using 
Theorem~\ref{thm simples of A} 
and the condensation results of Green, cf.~\ref{noth condense}, we 
will first determine
the simple $e_\catC A_{\catC,R}^\kappa e_\catC$-modules when $\catC$ is finite, 
cf.~Theorem~\ref{thm simples of RB(G,G)}. The remainder of this section is devoted 
to several consequences of Theorem~\ref{thm simples of RB(G,G)}.

One consequence, presented at the end of the section, will be on simple 
biset functors: recall from Proposition~\ref{prop fixed points} and 
Example~\ref{ex Bouc algebra}(c) that, when $\catC$ is finite, the maps $\alpha_{G,H}$, $G,H\in\Ob(\catC)$, induce an isomorphism
\begin{equation*}
  \alpha_\catC\colon \bigoplus_{G,H\in\Ob(\catC)}RB^\catC(G,H)\liso 
  e_\catC  A_{\catC,R}^\kappa e_\catC
\end{equation*}
of $R$-algebras. Thus, the simple biset functors for $\catC$ over $R$ are 
parametrized by the 
simple $A_{\catC,R}^\kappa$-modules $S$ with the property that 
$e_\catC\cdot S\neq \{0\}$, cf.~\ref{noth condense}. In 
Theorem~\ref{thm simple biset functors} we 
use the condensation functor approach to determine the simple biset functors on 
$\catB$, or on any subcategory $\catC$ of $\catB$ that satisfies the general 
assumptions in this section and is closed under taking subquotients in a sense 
defined in Corollary~\ref{cor section-closed 2}.

\smallskip
Recall from Theorem~\ref{thm simples of A} that,
in the case where $\catC$ is finite, the simple $A_{\catC,R}^\kappa$-modules are 
described as
\begin{equation*}
  S_{(P,K),T}:=\Hd(Ae'_{(P,K)}\otimes_{e'_{(P,K)}Ae'_{(P,K)}} \tilde{T})\,,
\end{equation*}
where $A=A_{\catC,R}^\kappa$,  $(P,K)$ is a section of some $G\in\Ob(\catC)$, $T$ is a simple
$R\Aut_\catC(P/K)$-module and $\tilde{T}$ is the same $R$-module as $T$ but viewed as an $e'_{(P,K)}Ae'_{(P,K)}$-module 
via the canonical isomorphism $R_\kappa\Gamma_{e_{(P,K)}}\cong R\Aut_\catC(P/K)$ and the inflation 
from $R_\kappa\Gamma_{e_{(P,K)}}$ to $e'_{(P,K)}Ae_{(P,K)}'$.

\smallskip
In the case where $\catC$ has only one object $G$ we will obtain results on simple 
modules of $RB^\catC(G,G)$ and, in particular, on simple modules of $RB(G,G)$. 
In \cite[Section~6.1]{BoucSLN}, 
Bouc shows that the isomorphism classes of simple $RB(G,G)$-modules can be 
parametrized by 
pairs consisting of an isomorphism class of a section $(P,K)$ of $G$ and an 
isomorphism class 
of a simple $R\Out(P/K)$-module $T$, where $\Out(P/K)$ denotes the outer 
automorphism group of 
$P/K$. More precisely, there exists an injective map from the set of isomorphism 
classes of 
simple $RB(G,G)$-modules $S$ to the set of such pairs. In this section we give a 
necessary 
condition for such a pair to belong to the image of the parametrization, i.e., we 
bound the 
image of this map from above, cf.~Theorem~\ref{thm simple RB(G,G)-modules}.

\begin{theorem}\label{thm simples of RB(G,G)}
Assume that the category $\catC$ is finite, that it satisfies 
Condition~(\ref{eqn closed}), 
and assume that for any two objects $G$ and $H$ of\/ $\catC$, the 
morphism set\/ $\catC_{G,H}$ 
is closed under $G\times H$-conjugation.
Let $G$ be an object of $\catC$, let $(P,K)$ be a section of $G$ 
such that $e_{(P,K)}\in\catC_{G,G}$, 
let $T$ be a simple $R\Aut_\catC(P/K)$-module, let $S_{(P,K),T}$ be the 
corresponding simple 
$A_{\catC,R}^\kappa$-module from Theorem~\ref{thm simples of A}, and let $\chi$ 
denote the character 
of $T$. Then, $e_\catC\cdot S_{(P,K),T}\neq \{0\}$ if and only if
there exist an object $H$ of $\catC$ and a morphism 
$L=(P'',K'',\tau,P',K')\in C_{H,H}$ with 
$q(L)=[P/K]$ such that 
\begin{equation}\label{eqn thm condition}
  \sum_{\substack{h\in H \\ (P',K') \link \lexp{h}{(P'',K'')}}}
  |K'\cap {}^hK''|\cdot \chi''(\tau\circ \beta(P',K';\lexp{h}{P''},\lexp{h}{K''}) \circ \bar{c}_h) \neq 0
\end{equation}
in the field $R$. Here $\chi''$ is the character of the simple 
$R\Aut_\catC(P''/K'')$-module 
$T''$ corresponding to  $T$ (cf.~Lemma~\ref{lem closed 2}) and 
$\bar{c}_h\colon P''/K''\liso \lexp{h}{P''}/\lexp{h}{K''}$ is the 
isomorphism induced by conjugation with $h$.
\end{theorem}

Before we prove the theorem, note that the argument 
$\tau\circ \beta(P',K';\lexp{h}{P''},\lexp{h}{K''}) \circ \bar{c}_h$ of $\chi''$ 
is actually 
an element of $\Aut_{\catC}(P''/K'')$, since it is the isomorphism $\eta_M$ for 
the group 
$M:=L*\lexp{(h,1)}{e_{(P'',K'')}}\in\catC_{H,H}$.

\bigskip
\begin{proof}
We set $e:=e_{(P,K)}$ and $e':=[G:K]e$.
By Lemma~\ref{lem Schur ann}, we have $e_\catC\cdot S\neq\{0\}$ if and only if 
$(e'Ae_\catC Ae')\cdot\tilde{T}\neq\{0\}$, and the latter is equivalent to the 
existence 
of an object $H$ of $\catC$ and morphisms $L\in\catC_{G,H}$ and $M\in\catC_{H,G}$ 
such that 
\begin{equation}\label{eqn nonzero 1}
  e'Le_HMe'\cdot\tilde{T}\neq\{0\}\,.
\end{equation} 
Let $\pi\colon e'Ae'= R_\kappa\Gamma_e\oplus RJ_e \to R_\kappa\Gamma_e$ denote the 
canonical projection. 
Since $RJ_e$ annihilates $\tilde{T}$, the condition in (\ref{eqn nonzero 1}) is 
equivalent 
to the condition $\pi(e'Le_HMe')\cdot T\neq\{0\}$. This in turn is equivalent to
\begin{equation*}
  e_\chi\cdot\pi(e'Le_HMe')\neq 0
\end{equation*}
in $R\Gamma_e$, where 
\begin{equation*}
  e_\chi:=\sum_{\sigma\in\Aut_\catC(P/K)}\chi(\sigma^{-1}) \cdot L_\sigma\,,\quad 
  \text{with } L_\sigma=(P,K,\sigma,P,K),
\end{equation*}
which is (up to a unit in $R$) the primitive central idempotent of
$R_\kappa\Gamma_e$ corresponding to $T$.
Replacing, if necessary, $L$ by $e*L$ and $M$ by $M*e$, we may assume that 
\begin{equation*}
  K\le k_1(L)\le p_1(L)\le P\quad\text{and}\quad K\le k_2(M)\le p_2(M)\le P\,.
\end{equation*}
Further, using Lemma~\ref{lem butterfly}, we may assume that  $K=k_1(L)=k_2(M)$ and 
$P=p_1(L)=p_2(M)$, since otherwise $\pi(e'Le_HMe')=0$. Altogether, since $e'=[G:K]e$, 
we obtain that $e_\catC\cdot S\neq \{0\}$ if and only if there exist an object $H\in\Ob(\catC)$ 
and morphisms $L\in\catC_{G,H}$ and $M\in\catC_{H,G}$ satisfying 
\begin{equation}\label{eqn nonzero 2}
  k_1(L)=K=k_2(M)\,, \quad p_1(L)=P=p_2(M)\,,
\end{equation}
and
\begin{equation}\label{eqn nonzero 3}
  e_\chi\cdot\pi(e'Le_HMe')\neq 0 \quad\text{in $R_\kappa\Gamma_e$.}
\end{equation}
Setting $(P',K'):=(p_2(L),k_2(L))$, $(P'',K''):=(p_1(M),k_1(M))$, $\phi:=\eta_L\colon P'/K'\liso P/K$, 
and $\psi:=\eta_M\colon P/K\liso P''/K''$, and assuming the conditions in (\ref{eqn nonzero 2}), 
we obtain, by Lemma~\ref{lem butterfly}, that
\begin{align*}
  \pi(e'Le_HMe') & = \sum_{h\in H}\pi(e'L\Delta_h(H)Me') \\
  & = \sum_{\substack{h\in H\\ (P',K') \link \lexp{h}{(P'',K'')}}} 
     \frac{|K'\cap{}^hK''|}{|H|^2} \cdot (P,K,\phi\circ\beta(P',K';\lexp{h}{P''},\lexp{h}{K''})\circ\bar{c}_h\circ\psi, P,K),
\end{align*}
where $\bar{c}_h\colon P''/K''\liso\lexp{h}{P''}/\lexp{g}{K''}$ is induced by the conjugation map 
$c_h$. Thus the condition in (\ref{eqn nonzero 3}) is equivalent to
\begin{equation*}
  \sum_{\sigma\in\Aut_\catC(P/K)} 
  \sum_{\substack{h\in H\\ (P',K') \link \lexp{h}{(P'',K'')}}} a_h\cdot \chi(\sigma^{-1}) \ \cdot\ 
     \bigl(\sigma\circ\phi\circ\beta(P',K';\lexp{h}{P''},\lexp{h}{K''})\circ\bar{c}_h\circ\psi\bigr) \neq 0
\end{equation*}
in $R\Aut_\catC(P/K)$, where
\begin{equation*}
   a_h=|K'\cap{}^h K''|\cdot\frac{|K|}{|H|^2\cdot |G|}\quad
   (\text{for $h\in H$ and $(P',K') \link \lexp{h}{(P'',K'')}$})\,. 
\end{equation*}
Applying the isomorphism 
$c_\psi\colon \Aut_\catC(P/K)\liso\Aut_\catC(P''/K'')$, $\sigma\mapsto \psi\circ\sigma\circ\psi^{-1}:=\sigma''$ 
and switching the summation translates the last condition into
\begin{equation*}
  \sum_{\substack{h\in H\\ (P',K') \link \lexp{h}{(P'',K'')}}} \sum_{\sigma''\in\Aut_\catC(P''/K'')} 
  |K'\cap{}^h K''|\cdot \chi''((\sigma'')^{-1})\ \cdot\
     \bigl(\sigma''\circ\psi\circ\phi\circ\beta(P',K';\lexp{h}{P''},\lexp{h}{K''})\circ\bar{c}_h\bigr) \neq 0
\end{equation*}
in $R\Aut_\catC(P''/K'')$, where $\chi''$ is the character of the simple
$R\Aut_\catC(P''/K'')$-module $T''$ arising from $T$ via $c_\psi$. Substituting 
$\theta''$ for
$\sigma''\circ\psi\circ\phi\circ\beta(P',K'';\lexp{h}{P''},\lexp{h}{K''})\circ\bar{c}_h$
and then switching the summation again, one obtains the equivalent condition
\begin{equation*}
  \sum_{\theta''\in\Aut_\catC(P''/K'')} 
  \sum_{\substack{h\in H\\ (P',K') \link \lexp{h}{(P'',K'')}}} 
   |K'\cap{}^h K''|\cdot \chi''\bigl((\theta'')^{-1}\circ\psi\circ\phi\circ\beta(P',K';\lexp{h}{P''},\lexp{h}{K''})\circ\bar{c}_h\bigr)
    \, \cdot\,    \theta''\  \neq\ 0\
\end{equation*}
in $R\Aut_\catC(P''/K'')$; note that here we used the fact that $\chi''$ is
a class function. 
Thus, $e_\catC\cdot S\neq \{0\}$ if and only if there exist
an object $H$ of $\catC$, sections $(P',K')$ and $(P'',K'')$ of $H$ that are isomorphic 
to $(P,K)$, and an isomorphism $\tau\colon P'/K'\to P''/K''$ with 
$(P'',K'',\tau,P',K')\in\catC_{H,H}$ such that
\begin{equation*}
  \sum_{\substack{h\in H\\ (P',K') \link \lexp{h}{(P'',K'')}}} 
   |K'\cap{}^h K''|\cdot \chi''\bigl(\tau\circ\beta(P',K'';\lexp{h}{P''},\lexp{h}{K''})\circ\bar{c}_h\bigr)\neq 0
\end{equation*} 
in $R$. This completes the proof of the theorem.
\end{proof}

We first derive an immediate corollary of the previous theorem.

\begin{corollary}\label{cor section-closed, abelian}
Assume that $\catC$ satisfies the same hypotheses as in 
Theorem~\ref{thm simples of RB(G,G)}, that $G$ is an object of $\catC$, $(P,K)$ is a section of $G$ such that $e_{P,K}\in\catC_{G,G}$, that $T$ is a simple $R\Aut_\catC(P/K)$-module, and that $S:=S_{(P,K), T}$ is the corresponding simple $A_{\catC,R}^\kappa$-module.

\smallskip
{\rm (a)}\, If there exists an object $H$ of $\catC$ such that $H$ and $P/K$ are isomorphic as groups and if $\Inn(P/K)$ acts trivially on $T$ 
then $e_{\catC}\cdot S\neq \{0\}$.

\smallskip
{\rm (b)}\, If $P/K$ is abelian then $e_\catC\cdot S\neq \{0\}$.
\end{corollary}

\begin{proof}
(a) The condition in 
(\ref{eqn thm condition}) is satisfied for $H$ and $L=\Delta(H)\in\catC_{H,H}$: 
in fact, the left-hand side of (\ref{eqn thm condition}) is equal to $\sum_{c\in\Inn(H)} \chi''(c)$. This element is non-zero if and only if $\Inn(H)$ acts trivially on $T''$. But this is equivalent to $\Inn(P/K)$ acting trivially on $T$. Now the result follows.

\smallskip
(b) The condition in (\ref{eqn thm condition}) is satisfied for $H=G$, $(P',K')=(P'',K'')=(P,K)$, and $\tau=\id_{P/K}$. In fact, the left-hand side is equal to $|G|\cdot|K|\cdot\chi''(\id_{P/K})$, which is clearly non-zero.
\end{proof}

Before we can derive further consequences from Theorem~\ref{thm simples of RB(G,G)} we need 
first an auxiliary result. Let $G$ be a finite group, let $\chi$ be the character of a 
finite-dimensional $RG$-module, and let $X\subseteq G$ be a subset. We set 
$X^+:=\sum_{x\in X} x\in RG$ and $\chi(X^+):=\sum_{x\in X}\chi(x)$. Note that if $H$ 
is a subgroup of $G$ then $\chi(H^+)$ is a non-zero multiple of the Schur inner 
product of the trivial character of $H$ and the restriction of $\chi$ to $H$. With 
this notation the following lemma holds. Its proof will appear in \cite{BK}. We repeat it here for the convenience of the reader.

\begin{lemma}\label{lem BK}
Let $\chi$ be the character of a finite-dimensional $RG$-module and assume that $
g\in G$ and $H\le G$ satisfy $\chi((gH)^+)\neq 0$. Then also $\chi(H^+)\neq 0$ and, 
in particular, the restriction of $\chi$ to $H$ contains the trivial character 
of $H$ as a constituent.
\end{lemma}

\begin{proof}
Let $\Delta\colon RG\to \Mat_n(R)$ be a representation affording the character $\chi$ and assume that $\chi(H^+)=0$. Then also $\chi(e_H)=0$, where $e_H=\frac{1}{|H|}H^+$ is an idempotent in $RG$. Since $e_H$ is an idempotent, the matrix $\Delta(e_H)$ is diagonalizable and its only eigenvalues are $0$ and $1$. Since $0=\chi(e_H)=\tr(\Delta(e_H))$, we obtain $\Delta(e_H)=0$. Thus, $\Delta(RG\cdot e_H)=0$. This finally implies that $\chi((gH)^+) = |H|\cdot \chi(ge_H)=0$, and the proof of the lemma is complete.
\end{proof}

For any section $(P,K)$ of a finite group $G$ we set $N_G(P,K):=N_G(P)\cap N_G(K)$ and define
\begin{equation*}
  \Aut_G(P/K):=\{\bar{c}_g\mid g\in N_G(P,K)\}\,.
\end{equation*}
Note that $\Inn(P/K)\le \Aut_G(P/K)\le \Aut(P/K)$. For each $G\in\Ob(\catC)$ and 
each section $(P,K)$ of $G$ with $e_{(P,K)}\in\catC_{G,G}$, one has $\Aut_G(P/K)\le \Aut_{\catC}(P/K)$. 
In fact, for each $g\in N_G(P,K)$ one has $\lexp{(g,1)}{e_{(P,K)}} = (P,K,\bar{c}_g,P,K)\in\catC_{G,G}$. 
Thus, we obtain
\begin{equation*}
  \Inn(P/K)\le \Aut_G(P/K)\le\Aut_\catC(P/K)\le \Aut(P/K)\,.
\end{equation*}
If also $H\in\Ob(\catC)$ and if $(P',K')$ is a section of $H$ with $P'/K'\cong P/K$ then 
there exists an isomorphism $\tau\colon P/K\myiso P'/K'$ such that $(P,K,\tau,P',K')\in\catC_{G,H}$, 
since $\catC$ satisfies (\ref{eqn closed}). Conjugation by $\tau$ induces an isomorphism 
$\Aut_{\catC}(P'/K')\to\Aut_{\catC}(P/K)$ that maps $\Aut_H(P'/K')$ to a subgroup 
$\tau\Aut_H(P'/K')\tau^{-1}$ of $\Aut_\catC(P/K)$ that also contains $\Inn(P/K)$. Different 
choices of $\tau$ yield conjugate subgroups of $\Aut_\catC(P/K)$. We denote by $\scrA_\catC(P/K)$ 
the set of subgroups of $\Aut_\catC(P/K)$ that arise this way. Thus, $\scrA_\catC(P/K)$ is a 
collection of subgroups of $\Aut_{\catC}(P/K)$ that is closed under $\Aut_\catC(P/K)$-conjugation 
and all of whose elements contain $\Inn(P/K)$.

\begin{corollary}\label{cor simples}
Assume that $\catC$ is a finite subcategory of $\catB$ that satisfies 
Condition~(\ref{eqn closed}), and assume that for any two objects $G$ and $H$ of\/ $\catC$, 
the morphism set\/ $\catC_{G,H}$ is closed under $G\times H$-conjugation.
Furthermore, let $G$ be an object of\/ $\catC$, let $(P,K)$ be a section of $G$ such that 
$e_{(P,K)}\in\catC_{G,G}$, let $T$ be a simple $R\Aut_{\catC}(P/K)$-module, and let $S=S_{(P,K),T}$ 
be the corresponding simple $A_{\catC,R}^\kappa$-module, cf.~Theorem~\ref{thm simples of A}. If 
$e_\catC\cdot S\neq\{0\}$ then there exists a subgroup $B$ of $\Aut_{\catC}(P/K)$ that belongs to 
$\scrA_{\catC}(P/K)$ such that the restriction of $T$ to $B$ has the 
trivial module as constituent. In particular, $\Inn(P/K)$ acts trivially on $T$.
\end{corollary}

\begin{proof}
By Theorem~\ref{thm simples of RB(G,G)}, there exist $H\in\Ob(\catC)$ and sections $(P',K')$ 
and $(P'',K'')$ of $H$, both isomorphic to $(P,K)$, and an isomorphism
$\tau\colon P'/K'\myiso P''/K''$ satisfying (\ref{eqn thm condition}). 
Note that if $h\in H$ satisfies $(P',K')\link\lexp{h}{(P'',K'')}$ then every element 
in the coset $hN_H(P'',K'')$ satisfies this condition. Thus, the condition in 
(\ref{eqn thm condition}) implies
\begin{equation*}
     \chi''\bigl((\sigma''\circ \Aut_H(P''/K''))^+\bigr) \neq 0
\end{equation*}
for some $\sigma''\in\Aut_{\catC}(P''/K'')$. Lemma~\ref{lem BK} implies that 
$\chi''(\Aut_H(P''/K'')^+)\neq 0$. Here, $\chi$ denotes the character of $T$ and $\chi''$ 
denotes the corresponding character of the group $\Aut_\catC(P''/K'')$. The last condition 
implies that $\chi(B^+)\neq 0$ for some $B\in\scrA_\catC(P/K)$. Thus, the restriction of 
$\chi$ to $B$ contains the trivial character as a constituent. Since $\Inn(P/K)\le B$, 
this also holds for $\chi$ restricted to $\Inn(P/K)$. Finally, since $\chi$ is the character 
of a simple $R\Aut(P/K)$-module and since $\Inn(P/K)$ is normal in $\Aut(P/K)$, this 
implies that $\Inn(P/K)$ acts trivially on $T$.
\end{proof}

If $\catC$ has only one object $G$ then we obtain results on the simple $RB^\catC(G,G)$-modules. Let $\calS$ 
denote a set of representatives for the isomorphism classes of sections $(P,K)$ of $G$ satisfying $e_{(P,K)}\in\catC$, and for each $(P,K)\in\calS$, let $\calT_{(P,K)}$ denote a set of representatives for the isomorphism classes of simple $R\Aut_\catC(P/K)$-modules. The simple $RB^\catC(G,G)$-modules are parametrized by pairs $((P,K),T)$ with $(P,K)\in\calS$ and $T\in\calT_{(P,K)}$ such that there exist sections $(P',K')$ and $(P'',K'')$ of $G$ and an isomorphism $\tau\colon P''/K''\liso P'/K'$ with $(P',K',\tau,P'',K'')$ such that the condition (\ref{eqn thm condition}) is satisfied. In this case we say that $((P,K),T)$ {\em parametrizes a simple $RB^\catC(G,G)$-module}.

\smallskip
The following theorem improves the results in \cite[Chapter~6]{BoucSLN} on the parametrization of simple $RB(G,G)$-modules in the case that $R$ is a field of characteristic $0$. It follows immediately from Corollary~\ref{cor simples} and Corollary~\ref{cor section-closed, abelian}(b).

\begin{theorem}\label{thm simple RB(G,G)-modules}
Assume that $G$ is a finite group and that $\catC$ is a submonoid of 
$(\scrS_{G\times G},*)$ satisfying the condition in (\ref{eqn closed}) and that $\catC$ is closed under $G\times G$-conjugation, and assume the notation from above. Let $(P,K)\in\calS$ and $T\in\calT_{P,K}$. 

\smallskip
{\rm (a)} If $((P,K),T)$ parametrizes a simple $RB^\catC(G,G)$-module 
then the restriction of $T$ to a subgroup $B$ of $\Aut_\catC(P/K)$ 
contains the trivial $RB$-module as a constituent; in particular, $\Inn(P/K)$ acts trivially on $T$.

\smallskip
{\rm (b)} If $P/K$ is abelian then $((P,K),T)$ parametrizes a simple $RB^\catC(G,G)$-module.
\end{theorem}

The following corollary is also a consequence of Corollary~\ref{cor section-closed, abelian}(a) and Corollary~\ref{cor simples}.

\begin{corollary}\label{cor section-closed 2}
Assume that $\catC$ satisfies Condition~(\ref{eqn closed}) and that for 
any two objects 
$G$ and $H$ of $\catC$ the set $\catC_{G,H}$ is closed under $G\times H$-conjugation.
Assume further that $\catC$ is finite and that the objects of $\catC$ are closed 
under taking subquotients in the following sense: for any object $G$ of $\catC$ and any section 
$(P,K)$ of $G$ with $e_{(P,K)}\in\catC_{G,G}$ there exists an 
object $H$ of $\catC$ with 
$H\cong P/K$.

Then the isomorphism classes of simple $e_\catC A_{\catC,R}^\kappa e_\catC$-modules 
are parametrized by pairs $(G, T)$, where $G$ runs through a set $\calS$ of representatives 
of the isomorphism classes  of groups $G$ occurring as objects of $\catC$ and, for each $G\in\calS$, $T$ runs through a set of representatives for the isomorphism classes of simple $R\Out_\catC(G)$-modules. Here, $\Out_\catC(G)$ is defined as $\Aut_\catC(G)/\Inn(G)$.
\end{corollary}

\begin{proof}
First note that, since $\catC$ is closed under taking subquotients, the set $\calS$ 
(in Theorem~\ref{thm simples of A}) 
of representatives for the isomorphism classes of sections $(P,K)$ 
of objects $G$ of $\catC$ satisfying $e_{(P,K)}\in\catC_{G,G}$ 
can be chosen such that $P=G$ and $K=1$. Next assume that if $T$ is a 
simple $R\Aut_\catC(G)$-module. If $e_\catC\cdot S=\{0\}$ for the corresponding 
simple $A_{\catC,R}^\kappa$-module $S$ then Corollary~\ref{cor simples} implies that $\Inn(G)$ acts trivially on $T$. Conversely, if $\Inn(G)$ acts trivially on $T$ then Corollary~\ref{cor section-closed, abelian}(a) implies that $e_\catC\cdot S\neq\{0\}$. This completes the proof of the corollary.
\end{proof}

\begin{remark}\label{rem simples}
Assume that $\catC$ satisfies Condition~(\ref{eqn closed}), that for any two objects 
$G$ and $H$ of $\catC$ the set $\catC_{G,H}$ is closed under $G\times H$-conjugation, and that the objects of $\catC$ are closed under taking subquotients in the sense of Corollary~\ref{cor section-closed 2}.

\smallskip
Corollary~\ref{cor section-closed 2} can be considered as a parametrization result for simple biset functors on $\catC$ over $R$, when $\catC$ is {\em finite} (cf.~Example~\ref{ex Bouc algebra}(c)).
To make the transition from $\catC$ being finite to $\catC$ being arbitrary,
it is straightforward to show that the standard techniques developed in \cite{BoucSLN}, 
especially those in Section 3.3, 4.2 and 4.3 that are used in the parametrization of 
simple biset functors, still work if one weakens the requirement of an 
{\em admissible subcategory}, cf.~\cite[Definition 4.1.3]{BoucSLN}, to requiring 
Condition~(\ref{eqn closed}), that the morphism sets $\catC_{G,H}$ are closed under 
$G\times H$-conjugation, and that the objects of $\catC$ are closed under taking sections 
in the sense of Corollary~\ref{cor section-closed 2}. These techniques allow to reduce the 
determination of simple biset functors on $\catC$ over $R$ to the case where $\catC$ is finite. 
Together with Corollary~\ref{cor section-closed 2} we then obtain 
the theorem stated below. There we use Bouc's definition 
of the simple biset functor $S_{(G,T)}$ associated to the pair $(G,T)$, where $G$ is an 
object of $\catC$ and $T$ is a simple $R\Out_\catC(G)$-module. Two pairs $(G,T)$ and 
$(G',T')$ are called {\em isomorphic} if $G$ and $G'$ are isomorphic and if for one 
(or equivalently each) $\tau\colon G'\to G$ with $(G,1,\tau,G',1)\in\catC_{G,G'}$ the 
modules $T'$ and $T$ correspond to each other under the map $c_\tau\colon\Aut_{\catC}(G')\to\Aut_{\catC}(G)$. 
Note that such an isomorphism $\tau$ exists. The parametrization in the following theorem 
slightly generalizes Bouc's parametrization in \cite[Theorem~4.3.10]{BoucSLN} where 
$\Aut_\catC(P/K)$ is assumed to be equal to $\Aut(P/K)$. Unfortunately, our methods 
require $R$ to be a field of characteristic $0$.
\end{remark}

\begin{theorem}\label{thm simple biset functors}
Let $\catC$ be a subcategory of\/ $\catB$, not necessarily finite. 
Assume that $\catC$ 
satisfies Condition~(\ref{eqn closed}), that for any two objects $G$ and 
$H$ of\/ $\catC$ 
the morphism set\/ $\catC_{G,H}$ is closed under $G\times H$-conjugation, 
and that\/ $\catC$ 
is closed under taking sections in the sense of Corollary~\ref{cor section-closed 2}.
Then the map 
$(G,T)\mapsto S_{(G,T)}$ induces a bijection between the set of isomorphism 
classes of pairs 
$(G,T)$, where $G$ is an object of\/ $\catC$ and $T$ is a simple 
$R\Out_\catC(G)$-module, 
and the set of isomorphism classes of simple biset functors for 
$\catC$ over $R$. 
\end{theorem}


\section{Cyclic Groups}\label{sec cyclic}

Throughout this section $R$ denotes a commutative ring and $\scrD$ denotes a
set of cyclic groups such that, for all $G\in\scrD$, the group order $|G|$ 
is a unit in $R$. By $\varphi$ we denote Euler's totient function.
The aim of this section is to show that if, in addition, also
$\varphi(|G|)$, for $G\in\scrD$, are units in $R$ then
the $R$-algebra $\bigoplus_{G,H\in\scrD}RB(G,H)$ is isomorphic
to the category algebra $R\catCtilde$, where $\catCtilde$ is the
finite inverse category associated with the set $\scrD$ that
was introduced in Definition \ref{defi group cat}. Thus we 
will, in particular, be able to invoke the results quoted in Section \ref{sec category algebras} 
to deduce that if $R$ is a field
such that, for every $G\in\scrD$, both $|G|$ and $\varphi(|G|)$ 
are non-zero in $R$ and if $\scrD$ is finite then $\bigoplus_{G,H\in\scrD}RB(G,H)$ is a semisimple 
$R$-algebra. In this case we 
will give an explicit decomposition of this algebra as a direct product of matrix algebras over group algebras.

We set 
\begin{equation*}
  B_R^\scrD:=\bigoplus_{G,H\in\scrD} RB(G,H)\quad\text{and}\quad
  \Atilde_R^\scrD:=\bigoplus_{G,H\in\scrD} R\scrS_{G\times H}\,.
\end{equation*}
These two $R$-modules are equipped with the following multiplications: for $G,H,H',K\in\scrD$, 
and for $a\in RB(G,H)$, $b\in RB(H',K)$, $x\in R\scrS_{G\times H}$, and $y\in R\scrS_{H'\times K}$, we have
\begin{equation*}
  a\cdot b := 
  \begin{cases} a\cdot_H b & \text{if $H=H'$,}\\
                        0 & \text{if $H\neq H'$,}
  \end{cases}
\end{equation*}
cf.~Example~\ref{ex Bouc algebra}(c), and
\begin{equation*}
  x\startilde^\kappa y:= 
  \begin{cases} x\startilde_H^\kappa y & \text{if $H=H'$,}\\
                         0 & \text{if $H\neq H'$.}
  \end{cases}
\end{equation*}
Note that $G\times H$ acts trivially on $R\scrS_{G\times H}$ when $G$ and 
$H$ are abelian. Therefore, Propositions~\ref{prop B(G)}, \ref{prop alpha compatible} 
and \ref{prop zeta} imply the following proposition.

\begin{proposition}\label{prop iso BD AtildeD}
The composition of the maps $(\alpha_{G,H})_{G,H\in\scrD}\colon B^\scrD_R\to\bigoplus_{G,H\in\scrD} R\scrS_{G\times H}$ 
and $(\zeta_{G,H})_{G,H\in\scrD}\colon \bigoplus_{G,H\in\scrD} R\scrS_{G\times H}\to \Atilde_R^\scrD$ 
defines an isomorphism 
\begin{equation*}
  \gamma\colon (B_R^\scrD,\cdot)\liso (\Atilde_R^\scrD,\startilde^\kappa)
\end{equation*}
of $R$-algebras.
\end{proposition}

Next we will focus on the structure of the $R$-algebra $(\Atilde_R^\scrD,\startilde^\kappa)$. 
First we will introduce a new $R$-basis of $\Atilde_R^\scrD$ and then, in 
Theorem~\ref{thm product of hats}, show that the product of any two of these 
basis elements is a multiple of a basis element, or is equal to 0. This will require several steps.

\begin{notation}
For finite groups $G$ and $H$ and a subgroup $L\le G\times H$, we set
\begin{equation*}
  \Ltilde:=\sum_{L'\in\scrP(L)} L'\in R\scrS_{G\times H}\,,
\end{equation*}
where
\begin{equation*}
  \scrP(L)=\{L'\le L\mid p_1(L')=p_1(L),\ p_2(L')=p_2(L)\}\,,
\end{equation*}
cf.~Definition~\ref{defi group cat}. Note that the elements in
$\{\Ltilde\mid L\le G\times H\}$, form again an $R$-basis of $R\scrS_{G\times H}$. We also set
\begin{equation*}
  s(L):=[p_1(L):k_1(L)] = [p_2(L):k_2(L)]\,.
\end{equation*}
Note that $s(L)=1$ if and only if $L=p_1(L)\times p_2(L)$ is a direct product.
\end{notation}

In order to see that the basis elements $\Ltilde$, $L\le G\times H$, have 
the desired property we first consider the special case of cyclic $p$-groups $G$ and $H$.

\begin{lemma}\label{lemma direct1}
Let $p$ be prime, and let $G$ and $H$ be
cyclic $p$-groups.

{\rm (a)}\, Let $M'\leq M\leq G\times H$ be such that $p_2(M')=p_2(M)$ and
$p_1(M')<p_1(M)$. Then $M=p_1(M)\times p_2(M)$;
in particular, $M'\leq U\times p_2(M)\leq M$, where
$U< p_1(M)$ is such that $[p_1(M):U]=p$.

{\rm (b)}\, Let $L'\leq L\leq G\times H$ be such that $p_1(L')=p_1(L)$ and
$p_2(L')<p_2(L)$. Then $L=p_1(L)\times p_2(L)$;
in particular, $L'\leq p_1(L)\times V\leq L$, where $V< p_2(L)$ is 
such that $[p_2(L):V]=p$.
\end{lemma}

\begin{proof}
We verify Part~(a); Part~(b) is proved analogously.
So let $(g,h)\in M$ be such that
$\langle g\rangle=p_1(M)$, thus $g\notin p_1(M')$. 
Since $h\in p_2(M)=p_2(M')$, there is some $g'\in p_1(M')$
such that $(g',h)\in M'$. Moreover,
$\langle g^{-1}g'\rangle=p_1(M)$, and $g^{-1}g'\in k_1(M)$. Hence
$k_1(M)=p_1(M)$, that is, $s(M)=1$.
\end{proof}

\begin{proposition}\label{prop hat}
Let $p$ be a prime that is invertible in $R$, let $G$, $H$, and $K$ be cyclic $p$-groups and let
$L\leq G\times H$ and $M\leq H\times K$ be subgroups such that $p_2(L)=p_1(M)$. Then,
in $R\scrS_{G\times K}$, one has
\begin{equation*}
  \Ltilde\startilde^\kappa_{H} \Mtilde=\begin{cases}
  \frac{\varphi(|k_2(L)\cap k_1(M)|)}{|H|}\cdot \widetilde{L*M}& \text{if $s(L)=1=s(M)$,}\\
  \frac{|k_2(L)\cap k_1(M)|}{|H|}\cdot \widetilde{L*M} &\text{otherwise.}
  \end{cases}
\end{equation*}
\end{proposition}

\begin{proof}
By the definition of $-\startilde^\kappa_{H}-$ in Definition \ref{def aLMN}
and by Corollary~\ref{cor aLMN}, we have
\begin{equation}\label{eqn 1}
  \Ltilde\startilde^\kappa_{H}\Mtilde 
  = \sum_{(L',M')\in \scrP(L)\times\scrP(M)}L'\startilde^  \kappa_{H}M'
  = \sum_{(L',M')\in\scrP(L)\times\scrP(M)} \
     \sum_{N\in\scrP(L'*M')} a^{L',M'}_N\cdot N\,,
\end{equation}
since $a^{L',M'}_N=0$ if $N\notin\scrP(L'*M')$.
Note that if $(L',M')\in\scrP(L)\times\scrP(M)$ then $p_2(L')=p_1(M')$, since 
$p_2(L)=p_1(M)$. Moreover, $p_2(L')=p_1(M')$ and $p_2(L)=p_1(M)$ together with 
Lemma~\ref{lem butterfly}(i) imply that 
\begin{equation*}
  p_1(L'*M')=p_1(L')=p_1(L)=p_1(L*M) \quad \text{and} \quad p_2(L'*M')=p_2(M')=p_2(M)=p_2(L*M)\,.
\end{equation*}
Therefore we have $\scrP(L'*M')\subseteq\scrP(L*M)$ and we can write
\begin{equation*}
  \Ltilde \startilde_H^\kappa \Mtilde = \sum_{N\in\scrP(L*M)} b^{L,M}_N \cdot N
\end{equation*}
with uniquely determined elements $b^{L,M}_N\in R$. Moreover, Equation~(\ref{eqn 1}) implies that, for 
$N\in\scrP(L*M)$, we have
\begin{equation*}
  b^{L,M}_N = \sum_{(L',M')\in\scrP(L)\times\scrP(M)} a^{L',M'}_N
  = \sum_{(L',M')\in\scrY^{L,M}_N\cap (\scrP(L)\times\scrP(M))} a^{L',M'}_N\,.
\end{equation*}
The last equation holds, since $a^{L',M'}_N=0$ if $(L',M')\notin\scrY^{L',M'}_N$ and $\scrY^{L',M'}_N\subseteq\scrY^{L,M}_N$.

Next we fix an element $N\in\scrP(L*M)$. To complete the proof of the proposition, 
it suffices to show that 
\begin{equation}\label{eqn bLMN}
  b^{L,M}_N = \begin{cases}
  \frac{\varphi(|k_2(L)\cap k_1(M)|)}{|H|} & \text{ if $s(L)=1=s(M)$,} \\
  \frac{|k_2(L)\cap k_1(M)|}{|H|} & \text{ otherwise.}
  \end{cases}
\end{equation}
Note that by M\"obius inversion in the poset $\scrY_N$ we have
\begin{equation}\label{eqn kappaLM}
  \kappa(L,M) 
  = \sum_{\substack{(L'',M'')\le(L',M')\\ \text{in } \scrY^{L,M}_N}} 
      \mu_{(L'',M''),(L',M')}^{\scrY_N}\cdot \kappa(L'',M'')
  = \sum_{(L',M')\in\scrY^{L,M}_N} a^{L',M'}_N
  = b^{L,M}_N + c^{L,M}_N
\end{equation}
with
\begin{equation*}
  c^{L,M}_N:=\sum_{(L',M')\in\scrY^{L,M}_N\smallsetminus (\scrP(L)\times\scrP(M))} a^{L',M'}_N\,.
\end{equation*}

\smallskip
{\em Claim:} If $L'\le L$ with $p_1(L')=p_1(L)$ and $p_2(L')<p_2(L)$ then $k_2(L')<k_2(L)$. In fact, by Equation~(\ref{eqn |L|}) we have $|L|= |p_1(L)|\cdot |k_2(L)|$ and $|L'|=|p_1(L')|\cdot|k_2(L')|$, and $p_2(L')<p_2(L)$ implies $L'<L$. Similarly, one can see that if $M'\le M$ with $p_2(M')=p_2(M)$ and $p_1(M')<p_1(M)$ then $k_1(M')<k_1(M)$. This settles the claim.

\smallskip
To finish the proof we distinguish three cases (not mutually exclusive, but covering all possibilities):

\smallskip
(i) If $s(L)>1$ or $s(M)>1$ then $\scrY^{L,M}_N\smallsetminus (\scrP(L)\times \scrP(M)) = \emptyset$ by Lemma~\ref{lemma direct1}. Therefore we obtain $c^{L,M}_N=0$ and Equation~(\ref{eqn kappaLM}) implies 
$b^{L,M}_N=\kappa(L,M)$, as desired.

\smallskip
(ii) If $k_2(L)=1$ or $k_1(M)=1$ then the above claim implies $c^{L,M}_N=0$, and Equation~(\ref{eqn kappaLM}) implies $b^{L,M}_N = \kappa(L,M) = 1/|H| = \varphi(|k_2(L)\cap k_1(L)|)/|H|$, as desired.

\smallskip
(iii) If $s(L)=1=s(M)$, $k_2(L)>1$ and $k_1(M)>1$ then, by Lemma~\ref{lemma direct1}, we have $\scrY^{L,M}_N\smallsetminus(\scrP(L)\times\scrP(M)) = \scrY^{p_1(L)\times V,V\times p_2(M)}_N$, where $V<p_2(L)$ is the unique subgroup of index $p$. Thus,
\begin{equation*}  
  c^{L,M}_N = 
  \sum_{(L',M')\in\scrY^{p_1(L)\times V,V\times p_2(M)}_N} a^{L',M'}_N
\end{equation*}
and this is equal to $\kappa(p_1(L)\times V,V\times p_2(M))=\kappa(L,M)/p$, by looking at the first and third term of the equations in (\ref{eqn kappaLM}) applied to $(p_1(L)\times V,V\times p_2(M))$. Now Equation~(\ref{eqn kappaLM}) implies
\begin{equation*}
  b^{L,M}_N = \kappa(L,M)-c^{L,M}_N = \kappa(L,M)-\kappa(L,M)/p = \frac{(1-1/p) \cdot |k_2(L)\cap k_1(M)|}{|H|}
  = \frac{\varphi(|k_2(L)\cap k_1(M)|)}{|H|}\,,
\end{equation*}
and the proof is complete.
\end{proof}

In the sequel we will denote by $\PP$ the set of prime numbers. For a finite abelian group 
$G$ and a prime $p$ we denote by $G_p$ the Sylow $p$-subgroup of $G$.

\begin{definition}
Assume that $G$, $H$, and $K$ are finite abelian groups such that $|H|$ 
is invertible 
in $R$. Let $L\le G\times H$ and $M\le H\times K$ be subgroups. For each $p\in\PP$, 
we define $\lambda_p(L,M)\in R$ as follows: if $p_2(L_p)\neq p_1(M_p)$ then we set 
$\lambda_p(L,M):=0$, and if $p_2(L_p)=p_1(M_p)$ then we set
\begin{equation*}
  \lambda_p(L,M):=\begin{cases}
    \frac{\varphi(|k_2(L_p)\cap k_1(M_p)|)}{|H_p|} & \text{if $s(L_p)=1=s(M_p)$,}\\
      \frac{|k_2(L_p)\cap k_1(M_p)|}{|H_p|} & \text{otherwise.}
    \end{cases}
\end{equation*}
Note that if $p$ does not divide $|H|$ then $\lambda_p(L,M)=1$. Moreover, we define
\begin{equation*}
  \lambda(L,M):=\prod_{p\in\PP} \lambda_p(L,M)\,.
\end{equation*}
Thus, $\lambda(L,M)\neq 0$ if and only if $p_2(L)=p_1(M)$, 
and in this case $\lambda(L,M)$ is invertible in $R$ provided that 
$\varphi(|H|)$ is also invertible in $R$.
\end{definition}

\begin{theorem}\label{thm product of hats}
Assume that $G$, $H$ and $K$ are finite cyclic groups such that $|H|$ is invertible 
in $R$ and let $L\le G\times H$ and $M\le H\times K$ be subgroups. Then
\begin{equation*}
  \Ltilde\startilde^\kappa_H\Mtilde =  \lambda(L,M)\cdot \widetilde{L*M}\,.
\end{equation*}
In particular, $\Ltilde\startilde^\kappa_H\Mtilde = 0$ if and only if $p_2(L)\neq p_1(M)$.
\end{theorem}

For the proof of Theorem~\ref{thm product of hats} we will need two lemmas that will 
reduce the theorem to the case of cyclic $p$-groups and therefore to Proposition~\ref{prop hat}. 
For a finite abelian group $G$, one has the canonical bijection 
\begin{equation*}
  \scrS_G\liso \bigtimes_{p\in\PP} \scrS_{G_p}\,,\quad U\mapsto (U_p)_{p\in\PP}\,,
\end{equation*}
which induces an $R$-module isomorphism
\begin{equation*}
  \theta_G\colon R\scrS_G \liso \bigotimes_{p\in\PP} R\scrS_{G_p}\,,
\end{equation*}
where the tensor product is taken over $R$. The inverse of $\theta$ maps the tensor 
product of a collection $U_p$, $p\in \PP$, of subgroups of $G$ to their product 
$\prod_{p\in\PP} U_p$ in $p$. If also $H$ is a finite abelian group then 
$(G\times H)_p=G_p\times H_p$ and we obtain an $R$-module isomorphism
\begin{equation*}
  \theta_{G,H}\colon R\scrS_{G\times H} \liso \bigotimes_{p\in\PP}R\scrS_{G_p\times H_p}\,.
\end{equation*}

The assertions of the following lemmas are straightforward verifications. 
The proof of the first one is left to the reader.

\begin{lemma}\label{lem p reduction 1}
Let $G$ and $H$ be finite abelian groups and let $L\le G\times H$ be a subgroup. Then the following equations hold in $\bigotimes_{p\in\PP} R\scrS_{G_p\times H_p}$:

\smallskip
{\rm (a)}\, $\theta_{G,H}(\sum_{L'\leq L}L')=\otimes_{p\in\mathbb{P}}(\sum_{L_p'\leq L_p}L_p')$;

\smallskip
{\rm (b)}\, $\theta_{G,H}(\Ltilde)= \otimes_{p\in\PP} \widetilde{L_p}$.
\end{lemma}

\begin{lemma}\label{lem p reduction 2}
Let $G$, $H$ and $K$ be finite abelian groups such that $|H|$ is invertible in $R$. 
Then the diagram
\begin{diagram}[90]
  \movevertex(-135,-10){R\scrS_{G\times H}\otimes R\scrS_{H\times K}} & 
             \movearrow(0,-5){\Ear[150]{-*_H^\kappa-}} & \movevertex(100,-10){R\scrS_{G\times K}} &&
  \movevertex(-120,-80){\sar[220]{\ \wr\ \xi_{G,H,K}\circ(\theta_{G,H}\otimes\theta_{H,K})}}&\movevertex(150,-80){\saR[220]{\wr\ \theta_{G,K}}}&&
  \movevertex(-110,30){\Sear[30]{\zeta_{G,H}\otimes\zeta_{H,K}}}& \movevertex(110,30){\Swar[30]{\zeta_{G,K}}} &&
  \movevertex(-50,30){R\scrS_{G\times H}\otimes R\scrS_{H\times K}} & 
             \movearrow(20,15){\Ear[100]{-\startilde_H^\kappa-}} & \movevertex(60,30){R\scrS_{G\times K}} &&
  \movevertex(-50,20){\saR[65]{\wr\ \xi_{G,H,K}\circ(\theta_{G,H}\otimes\theta_{H,K})}} & &
          \movearrow(60,10){\saR[65]{\wr\ \theta_{G,K}}}&&
  \movevertex(-57,0){\mathop{\bigotimes}\limits_{p\in\PP}\bigl(R\scrS_{G_p\times H_p} \otimes 
             R\scrS_{H_p\times K_p}\bigr)}  &
          \movearrow(18,0){\eaR[80]{\mathop{\bigotimes}\limits_{p\in\PP}(-\startilde_{H_p}^\kappa -)}} & 
          \movevertex(60,0){\mathop{\bigotimes}\limits_{p\in\PP} R\scrS_{G_p\times K_p}} &&
  \movevertex(-90,0){\neaR[30]{\mathop{\bigotimes}\limits_{p\in\PP}(\zeta_{G_p,H_p}\otimes\zeta_{H_p,K_p})}}
              \movevertex(140,0){\nwaR[30]{\mathop{\bigotimes}\limits_{p\in\PP}\zeta_{G_p,K_p}}}&&
  \movevertex(-120,0){\mathop{\bigotimes}\limits_{p\in\PP}\bigl(R\scrS_{G_p\times H_p} \otimes 
             R\scrS_{H_p\times K_p}\bigr)}  &
          \movearrow(10,0){\eaR[150]{\mathop{\bigotimes}\limits_{p\in\PP}(-*_{H_p}^\kappa -)}} & 
          \movevertex(100,0){\mathop{\bigotimes}\limits_{p\in\PP} R\scrS_{G_p\times K_p}} &&
\end{diagram}
\bigskip

\noindent
is commutative. Here 
\begin{equation*}
  \xi_{G,H,K}:\bigl(\bigotimes_{p\in\mathbb{P}}R\scrS_{G_p\times H_p}\bigr)\otimes
  \bigl(\bigotimes_{p\in\mathbb{P}}R\scrS_{H_p\times K_p}\bigr)\myiso 
  \bigotimes_{p\in\PP}\left(R\scrS_{G_p\times H_p}\otimes R\scrS_{H_p\times K_p}\right)
\end{equation*}
denotes the canonical isomorphism.
\end{lemma}

\begin{proof}
Let $L\le G\times H$ and $M\le H\times K$. 
The two paths in the outer square map 
$L\otimes M$ to the same element since
\begin{equation*}
  \frac{|k_2(L)\cap k_1(M)|}{|H|} 
  = \prod_{p\in\PP} \frac{|(k_2(L)\cap k_1(M))_p|}{|H_p|}
  = \prod_{p\in\PP} \frac{|k_2(L_p)\cap k_1(M_p)|}{|H_p|}\,.
\end{equation*}
Furthermore, the top and bottom rectangles of the 
diagram commute, by Proposition \ref{prop zeta}, and the left and right rectangles of the diagram commute, by
Lemma \ref{lem p reduction 1}(a). Lastly, since the maps $\zeta$ with various indices are isomorphisms, also the inner square commutes.
\end{proof}

\begin{proof} {\em of Theorem~\ref{thm product of hats}}. \quad 
Consider 
the inner square of the diagram in Lemma~\ref{lem p reduction 2}, and 
$\Ltilde\otimes \Mtilde\in R\scrS_{G\times H}\otimes R\scrS_{H\times K}$ in its 
top left corner. Using Lemma~\ref{lem p reduction 1}(b), 
the left-hand arrow in this inner square
 applied to this element yields the tensor product 
of the elements $\widetilde{L_p}\otimes \widetilde{M_p}$. By Proposition~\ref{prop hat}, 
this latter element is mapped under the bottom map to the tensor product 
of the elements $\lambda_p(L,M)\cdot (\widetilde{L_p*M_p})$. Since $L_p*M_p=(L*M)_p$, 
also $\lambda(L,M)\cdot \widetilde{L*M}$ is mapped to this element under $\theta_{G,K}$, 
by Lemma~\ref{lem p reduction 1}. This completes the proof of the theorem.
\end{proof}

Theorem~\ref{thm product of hats} allows us to view the $R$-algebra $\Atilde_R^\scrD$ 
as a twisted category algebra $R_\lambda\catCtilde$: let $\catCtilde=\catCtilde(\scrD)$ denote the 
inverse category associated to $\scrD$ as in 
Definition~\ref{defi group cat}, and
recall that we are assuming $|G|$ 
to be invertible in $R$, for every $G\in\scrD$. 
Recall further that the objects of $\catCtilde$ are pairs $(G,G')$ 
with $G\in\scrD$ and $G'\le G$. Moreover, for objects $(G,G')$ and $(H,H')$ of $\catCtilde$, 
the set of morphisms from $(H,H')$ to $(G,G')$ is given by $\scrP(G'\times H')$, and if 
$(K,K')$ is also an object of $\catCtilde$ and $L\in\scrP(G'\times H')$ and $M\in\scrP(H'\times K')$ 
are morphisms in $\catCtilde$ then their composition is defined by $L*M$. The identity 
morphism of $(G,G')$ is equal to $\Delta(G')$. Recall also that, for a morphism $L$ between $(H,H')$ 
and $(G,G')$, we will often write $\GLH$ to indicate 
that it is a morphism between the objects $(H,H')$ and $(G,G')$. The associativity of 
the $R$-algebra $\Atilde_R^\scrD$ and Theorem~\ref{thm product of hats} imply that 
$(L,M)\mapsto \lambda(L,M)$ defines a $2$-cocycle on $\catC$ with values in $R^\times$ (if also $\varphi(|G|)$ is invertible in $R$ for all $G\in\scrD$), 
and that the maps 
\begin{equation*}
  R\scrS_{G\times H}\to R_{\lambda}\catCtilde\,,\quad 
  \Ltilde\mapsto L\in\Hom_{\catCtilde}\bigl((H,p_2(L)),(G,p_1(L))\bigr)\,,
\end{equation*}
for $G,H\in\scrD$, induce an $R$-algebra isomorphism
\begin{equation}\label{eqn delta}
  \delta\colon \Atilde_R^\scrD\liso R_\lambda\catCtilde
\end{equation}
onto the twisted category algebra $R_\lambda\catCtilde$ of the inverse category $\catCtilde$.

\medskip
Our next goal is to show that $R_\lambda\catCtilde$ is isomorphic to the \lq untwisted\rq \,
category algebra $R\catCtilde$. The next proposition will achieve this by showing 
that $\lambda$ is a coboundary.

\begin{definition}\label{def mu}
For objects $(G,G')$, $(H,H')$ of $\catCtilde$ and a morphism $\GLH \in\scrP(G'\times H')$ 
from $(H,H')$ to $(G,G')$ we set
\begin{equation*}
  \mu(\GLH):=\prod_{p\in\PP} \mu_p(\GLH)\,,
\end{equation*}
where 
\begin{equation*}
  \mu_p(\GLH):=\begin{cases}
    \frac{\varphi(|p_1(L_p)|)}{|H_p|} & \text{if $s(L_p)=1$,}\\
    \frac{|k_1(L_p)|}{|H_p|} & \text{if $s(L_p)>1$.}
  \end{cases}
\end{equation*}
\end{definition}

\begin{proposition}\label{prop cochain mu}
Assume that, for each $G\in\scrD$, 
both $|G|$ and $\varphi(|G|)$ are invertible in $R$.
Let $(G,G'),(H,H'),(K,K')\in\Ob(\catCtilde)$, and let
$\GLH\in\Hom_{\catCtilde}((H,H'),(G,G'))$ and $\HMK\in\Hom_{\catCtilde}((K,K'),(H,H'))$.
Then one has
\begin{equation*}
  \lambda(\GLH,\HMK)=\mu(\GLH)\cdot\mu(\HMK)\cdot\mu(\GLMK)^{-1}\,.
\end{equation*}
Thus the $2$-cocycle $\lambda$ on the category $\catCtilde$ is a $2$-coboundary, and the 
twisted category algebra $R_\lambda\catCtilde$ is isomorphic to the
category algebra $R\catCtilde$, via the isomorphism
\begin{equation*}
  \epsilon\colon  R_\lambda\catCtilde\to R\catCtilde\,,\quad \GLH\mapsto \mu(\GLH)\cdot \GLH\,,
\end{equation*}
for $\GLH\in\Mor(\catCtilde)$.
\end{proposition}

\begin{proof}
For convenience, set $L:={}_GL_H$ and $M:={}_HM_K$.
By the definitions of $\lambda$ and $\mu$, it suffices to show that
\begin{equation}\label{eqn lambdamu}
\lambda_p(L_p,M_p)\mu_p(L_p)^{-1}\mu_p(M_p)^{-1}
=\mu_p(L_p*M_p)^{-1}\,,
\end{equation}
for all $p\in\mathbb{P}$. To this end, let $p\in\mathbb{P}$, and set
$l_p:=\lambda_p(L_p,M_p)\mu_p(L_p)^{-1}\mu_p(M_p)^{-1}$ and $r_p:=\mu_p(L_p*M_p)^{-1}$.
We distinguish four cases:

\smallskip
(i)\, If $s(L_p)=1=s(M_p)$ then we have $L_p*M_p=p_1(L_p)\times p_2(M_p)$, and thus
\begin{equation*}
  l_p=\frac{|G_p|}{\varphi(|p_1(L_p)|)}=\frac{|G_p|}{\varphi(|p_1(L_p*M_p)|)}=r_p\,.
\end{equation*}

\smallskip
(ii)\, If $s(L_p)=1$ and $s(M_p)>1$ then $L_p*M_p=p_1(L_p)\times p_2(M_p)$, and
\begin{equation*}
  l_p=\frac{|G_p| \cdot |H_p| \cdot |k_1(M_p)|}{|H_p| \cdot \varphi(|p_1(L_p)|) \cdot |k_1(M_p)|} =
  \frac{|G_p|}{\varphi(|p_1(L_p*M_p)|)}=r_p\,.
\end{equation*}

\smallskip
(iii)\, If $s(L_p)>1$ and $s(M_p)=1$ then $L_p*M_p=p_1(L_p)\times p_2(M_p)$, and
\begin{equation*}
  l_p=\frac{|G_p| \cdot |H_p| \cdot |k_2(L_p)|}{|H_p| \cdot \varphi(|p_2(L_p)|) \cdot |k_1(L_p)|}\,.
\end{equation*}
Since $s(L_p)>1$, we also have $|p_i(L_p)|>1$, for $i=1,2$, and hence
$\varphi(|p_i(L_p)|)=|p_i(L_p)|-|p_i(L_p)|/p$, for $i=1,2$. Thus
\begin{equation*}
  \frac{\varphi(|p_2(L_p)|)}{|k_2(L_p)|}=\frac{|p_2(L_p)|}{|k_2(L_p)|}-\frac{|p_2(L_p)|}{p \cdot |k_2(L_p)|}
  =\frac{|p_1(L_p)|}{|k_1(L_p)|}-\frac{|p_1(L_p)|}{p \cdot |k_1(L_p)|}=\frac{\varphi(|p_1(L_p)|)}{|k_1(L_p)|}\,.
\end{equation*}
This shows that
\begin{equation*}
  l_p=\frac{|G_p|}{\varphi(|p_1(L_p)|)}=\frac{|G_p|}{\varphi(|p_1(L_p*M_p)|)}=r_p\,.
\end{equation*}

\smallskip
(iv)\, If $s(L_p)>1$ and $s(M_p)>1$ then, by Lemma \ref{lem butterfly},
$s(L_p*M_p)=\min\{s(L_p),s(M_p)\}>1$, since $p_2(L_p)=p_1(M_p)$.
Thus we have
\begin{equation*}
  l_p=\frac{|G_p| \cdot |H_p| \cdot |k_2(L_p)\cap k_1(M_p)|}{|H_p| \cdot |k_1(M_p)| \cdot |k_1(L_p)|} 
  =\frac{|G_p| \cdot |k_1(L_p)| \cdot |k_1(M_p)|}{|k_1(L_p*M_p)| \cdot |k_1(L_p)| \cdot |k_1(M_p)|} 
  =\frac{|G_p|}{|k_1(L_p*M_p)|}=r_p\,.
\end{equation*}

\smallskip
This settles Equation (\ref{eqn lambdamu}), and the assertion of the 
proposition follows.
\end{proof}

In consequence of Proposition~\ref{prop iso BD AtildeD}, Equation~(\ref{eqn delta}),
Proposition \ref{prop cochain mu},
Theorem \ref{thm matrix}, and Proposition \ref{prop group cat}
we have established the following theorem.

\begin{theorem}\label{thm A matrix dec}
Let $\scrD$ be a finite set of cyclic groups,
let $\catCtilde$ be the finite inverse category associated with $\scrD$ in Definition~\ref{defi group cat}, 
and let $R$ be a commutative ring.
Suppose that, for all $G\in\scrD$, both
$|G|$ and $\varphi(|G|)$ are units in $R$. Then there is a sequence of $R$-algebra
isomorphisms
\begin{equation*}
  \bigoplus_{G,H\in\scrD}RB(G,H)= B^\scrD_R \Ar{\gamma} \Atilde^\scrD_R
  \Ar{\delta} R_\lambda\catCtilde \Ar{\epsilon} R\catCtilde \Ar{\omega} \bigtimes_k \Mat_{n(k)}(R\Aut(C_k))\,;
\end{equation*}
here $k$ varies over all divisors of all group orders $|G|$ ($G\in\scrD$), 
$C_k$ denotes a cyclic group of order $k$, and $n(k)$ the number of
sections of groups in $\scrD$ isomorphic to $C_k$. In particular, if $R$ is a field then
$B^\scrD_R$ is a semisimple $R$-algebra.
\end{theorem}

\begin{remark}\label{rem A matrix dec}
Suppose that $R$ is a field.

\smallskip
(a) Specializing to the case where the set $\scrD$ consists of 
one cyclic group $G$ only, Theorem \ref{thm A matrix dec} implies
that the double Burnside algebra $RB(G,G)$ is semisimple, provided
$|G|$ and $\varphi(|G|)$ are non-zero in $R$. Suppose that, whenever $p$ is
a prime divisor of $|G|$, also $p^2$ is a divisor of $|G|$. 
In this case $\varphi(|G|)\in R^\times$ implies $|G|\in R^\times$,
and we recover one half of \cite[Proposition 6.1.7]{BoucSLN}, characterizing the 
fields $R$ and finite groups $G$ such that $RB(G,G)$ is a semisimple $R$-algebra. 
In addition, Theorem~\ref{thm A matrix dec} yields an explicit decomposition. 

\smallskip
(b) Assume now that $R$ has characteristic $0$. Note that the $R$-algebra $R\catCtilde$ and the corresponding isomorphism $\omega$ from Theorem~\ref{thm matrix} are 
defined for an arbitrary finite set $\scrD$ of groups. However, if $\scrD$ does contain a non-cyclic group $G$ then $R\catCtilde$ cannot be isomorphic to $\Atilde^{\scrD}_R$, since otherwise $RB(G,G)$ (which can be considered as a condensed algebra of $\Atilde^{\scrD}_R$) would be semisimple, contradicting Bouc's characterization, cf.~\cite[Proposition~6.1.7]{BoucSLN}. See also \cite{Barker} for related results on the semisimplicity of biset functors in characteristic $0$.
\end{remark}

\begin{remark}\label{rem RC basis}
Recall from Remark \ref{rem omega} the
explicit decomposition of $R\catCtilde$ as a direct sum of two-sided ideals $I_e$
that correspond to the matrix algebras under the above isomorphism. 
We close this section by determining the $R$-basis of $R\catCtilde$ corresponding under $\omega$ to the
obvious $R$-basis in the direct product of matrix rings over group algebras, consisting of matrices that have only one non-zero entry, namely a group element in the group algebra. 
We emphasize that in Remark \ref{rem omega}, which explicitly described the isomorphism $\omega$, $R$ could be any commutative ring; the additional assumptions in Theorem \ref{thm A matrix dec} were only necessary to establish the isomorphism $\epsilon\circ\delta\circ\zeta\circ\alpha\colon B^\scrD_R\cong R \catCtilde$.
For this purpose it is convenient to use a more arithmetical notation for the subgroups $L$ of $G\times H$ in the case that $G$ and $H$ are cyclic groups. 
\end{remark}

\begin{notation}\label{not subgroups G}
Let $G=\langle x \rangle$ and $H=\langle y\rangle$ be cyclic groups with $|G|=n$ and $|H|=m$. Suppose that $a, b,k \in\mathbb{N}$ are
such that $a\mid n$, $b\mid m$, and $k\mid\gcd\{a,b\}$. Then 
$x^{n/a}\langle x^{nk/a}\rangle$ and $y^{m/b}\langle y^{mk/b}\rangle$ 
generate subquotients of order $k$ of $G$ and $H$, respectively.
Moreover, let $i\in\mathbb{Z}$ be such that $\gcd\{k,i\}=1$. Then the map
\begin{equation*}
  \alpha_i:\langle y^{m/b}\rangle/\langle y^{mk/b}\rangle\to \langle x^{n/a}\rangle/\langle x^{nk/a}\rangle\,,\quad 
  y^{m/b}\langle y^{mk/b}\rangle\mapsto x^{in/a}\langle x^{nk/a}\rangle\,,
\end{equation*}
defines a group isomorphism. We denote the group
$(\langle x^{n/a}\rangle,\langle x^{nk/a}\rangle,\alpha_i,\langle y^{m/b}\rangle,\langle y^{mk/b}\rangle)$ of $G\times H$ by 
\begin{equation*}
  {}_G(k;a,i,b)_H\,.
\end{equation*}
Thus, if $L={}_G(k;a,i,b)_H$ then one has $|p_1(L)|=a$, $|k_1(L)|=a/k$, $|p_2(L)|=b$, $|k_2(L)|=b/k$, $s(L)=k$. Clearly, each subgroup $L$ of $G\times H$ can be written in this way with $k$, $a$, and $b$, uniquely determined by $L$. Moreover, ${}_{G}(k;a,i,b)_{H}={}_{G}(k;a,j,b)_{H}$ if and only if $i,j\in\ZZ$ satisfy $i\equiv j\pmod{k}$.

Note that the subgroups $\langle y^{m/b}\rangle$, $\langle y^{mk/b}\rangle$, $\langle x^{n/a}\rangle$, and $\langle x^{nk/a}\rangle$ do not depend on the choice of the generators $x$ and $y$. However, the isomorphism $\alpha_i$ does depend on these choices. Thus, when we will use the notation ${}_G(k;a,i,b)_H$ we assume that the groups $G$ and $H$ are equipped with a fixed choice of generators.
\end{notation}

With this notation the following lemma is an immediate consequence of Lemma \ref{lem butterfly}; we thus omit its proof.

\begin{lemma}\label{lemma star G}
With the notation from \ref{not subgroups G}, let $G$, $H$, and $K$
be cyclic groups,
and let $L\leq G\times H$ and $M\leq H\times K$ with $p_2(L)=p_1(M)$ be given by
\begin{equation*}
  L={}_{G}(k;a,i,b)_{H}\quad \text{and}\quad M={}_{H}(l;b,j,c)_{K}\,,
\end{equation*}
for divisors $a \mid |G|$, $b \mid |H|$, $c\mid |K|$, $k\mid\gcd\{a,b\}$,
$l\mid\gcd\{b,c\}$, and integers $i,j\in\ZZ$ with $\gcd\{i,k\}=1=\gcd\{j,l\}$. Then
\begin{equation*}
  L*M={}_{G}(\gcd\{k,l\};a,ij,c)_{K}\,.
\end{equation*}
\end{lemma}

\begin{remark}
Recall from Proposition \ref{prop group cat}
that the idempotent endomorphisms in $\catCtilde$ are
of the form ${}_G(k;a,1,a)_G$, for $G\in\scrD$, $a\mid |G|$ and $k\mid a$.
Moreover, idempotent endomorphisms ${}_G(k;a,1,a)_G$ and ${}_H(l;b,1,b)_H$
in $\catCtilde$ are equivalent if and only if $l=k$.
\end{remark}

Let $\mu\colon \NN\to\NN$ denote the number-theoretic M\"obius function given by $\mu(n)=0$ if $n$ is divisible by $p^2$ for some prime $p$, and $\mu(n)=(-1)^r$ if $n$ is the product of $r$ pairwise distinct primes. Recall that if one considers the divisors $l$ of a fixed natural number $k$ as a partially ordered set under the divisibility relation then the M\"obius function for this poset, applied to $(l,k)$, is equal to $\mu(k/l)$.

\begin{lemma}\label{lemma I_e}
{\rm (a)}\, For every idempotent endomorphism $e:={}_G(k;a,1,a)_G$ 
in $\catCtilde$, one has
\begin{equation*}
  \{s\in\Mor(\catCtilde)\mid \shat\circ s=e\}=\{{}_H(k;b,i,a)_G\mid H\in\scrD, k\mid b, b\mid |H|,\, \gcd\{i,k\}=1\}\,.
\end{equation*}

\smallskip
{\rm (b)}\, For every $s:={}_G(k;a,i,b)_H\in\Mor(\catCtilde)$,
one has
\begin{equation*}
  \{u\leq s\mid u\in \Mor(\catCtilde)\}=\{{}_G(l;a,i,b)_H\mid l\mid k\}\,.
\end{equation*}
In particular, for $u:={}_G(l;a,i,b)_H$ with $u\le s$ in $\Mor(\catCtilde)$, one has $\mu_{u,s}^{\Mor(\catCtilde)}=\mu(k/l)$.
\end{lemma}

\begin{proof}
Both parts of the lemma are easy 
consequences of \ref{noth inverse cat}, Proposition \ref{prop group cat},
and Lemma \ref{lemma star G}. 
\end{proof}

\begin{notation}\label{not isoclasses idempt}
Suppose that $k\in\mathbb{N}$ is such that $k$ divides $|G|$, for some $G\in\scrD$.
Let $\{e_1,\ldots,e_{n(k)}\}$ be the equivalence class of the idempotent
endomorphism ${}_G(k;k,1,k)_G$ in $\catCtilde$. Thus, for each $i=1,\ldots,n(k)$, one has a pair $(G_i,a_i)$, where $G_i\in\scrD$ and $a_i\in\NN$ such that $a_i$ divides $|G_i|$, $k$ divides $a_i$, and
$e_i={}_{G_i}(k;a_i,1,a_i)_{G_i}$.
\end{notation}

With this notation, the following result is now a consequence of Lemma \ref{lemma I_e} and Remark \ref{rem omega}. It makes the map $\omega$ in Theorem~\ref{thm A matrix dec} as explicit as possible.

\begin{corollary}\label{cor ideals}
Let $\scrD$ be a finite set of cyclic groups, let $\catCtilde$ be
the finite inverse category associated with $\scrD$ in Definition~\ref{defi group cat}, and let
$R$ be any commutative ring. Moreover, let $k$ be a divisor of the order of some $G\in\scrD$ and assume the notation from \ref{not isoclasses idempt}.
For $i,j\in\{1,\ldots,n(k)\}$, and for an integer $t$ with $\gcd\{t,k\}=1$, set
\begin{equation*}
  b(k; i,t,j):=\sum_{l\mid k} \mu(k/l)\cdot {}_{G_i}(l;a_i,t,a_j)_{G_j}\in R\catCtilde\,.
\end{equation*}
Then $\omega(b(k;i,t,j))$ is the standard basis element in $\bigtimes_k \Mat_{n(k)}(R\Aut(C_k))$ that has zero entry everywhere except in the $k$-th component, and whose $k$-th component is equal to the matrix whose only non-zero entry is located in the $(i,j)$-th position and is equal to the element in $\Aut(C_k)$ that corresponds to $t$. In particular, the central idempotent $e_k$ of $R\catCtilde$, corresponding under $\omega$ to the tuple consisting of the identity matrix in the $k$-th component and the zero matrix in all other components, is equal to
\begin{equation*}
  \sum_{i=1}^{n(k)}\sum_{l\mid k}\mu(k/l) \cdot  {}_{G_i}(l;a_i,1,a_i)_{G_i}\,.
\end{equation*}
\end{corollary}



\begin{thebibliography}{00}
\bibitem[Ba]{Barker} {\sc L.~Barker:} 
         Rhetorical biset functors, rational $p$-biset functors and their semisimplicity in characteristic zero. 
         J.\ Algebra {\bf 319} (2008), 3810--3853.

\bibitem[BD]{BD11} {\sc R.~Boltje, S.~Danz:} 
      A ghost ring for the left-free double Burnside ring and an application to fusion systems.
      Adv. Math. \textbf{229} (2012), 1688--1733.

\bibitem[BK]{BK} {\sc R.~Boltje, B. K\"ulshammer:} Central idempotents in double Burnside rings, in preparation.

\bibitem[Bc1]{BoucJA} {\sc S.~Bouc:} 
            Foncteurs d'ensembles munis d'une double action. 
            J.\ Algebra~{\bf 183} (1996), 664--736.

\bibitem[Bc2]{Bouc2007} {\sc S.~Bouc:} 
            The functor of units of Burnside rings for $p$-groups.  
            Comment. Math. Helv.  {\bf 82}  (2007), 583--615. 

\bibitem[Bc3]{BoucSLN} {\sc S.~Bouc:} Biset functors for finite groups.
Lecture Notes in Mathematics, 1990. Springer-Verlag, Berlin, 2010.

\bibitem[BST]{BST} {\sc S.~Bouc, R.~Stancu, J.~Th\'evenaz:} 
         Simple biset functors and double Burnside rings. Preprint 2012. 
         arXiv:1203.0195. 

\bibitem[BT1]{BoucThevenaz00} {\sc S.~Bouc, J.~Th\'evenaz:} 
          The group of endo-permutation modules. 
          Invent. Math. {\bf 139} (2000), 275--349.

\bibitem[BT2]{BoucThevenaz2010} {\sc S.~Bouc, J.~Th\'evenaz:} Stabilizing bisets. 
Adv. Math. {\bf 229} (2012), 1610--1639.

\bibitem[BLO]{BLO1} {\sc C.~Broto, R.~Levi, B.~Oliver:} 
     The homotopy theory of fusion systems.
      J.\ Amer.\ Math.\ Soc.~{\bf 16} (2003), 779--856.

\bibitem[D]{D} {\sc A.~W.~M. Dress:} 
      A characterization of solvable groups. 
      Math. Z. {\bf 110} (1969), 213--217.

\bibitem[Gr1]{Gr1} {\sc J.~A.~Green:} On the structure of semigroups. Ann.
of Math. (2) {\bf 54} (1951), 163--172.

\bibitem[Gr2]{Gr2} {\sc J.~A.~Green:} Polynomial representations of
$\mathrm{GL}_n$. Second corrected and augmented edition.
 With an appendix on Schensted correspondence and Littelmann paths by K.
Erdmann, Green and M. Schocker.
 Lecture Notes in Mathematics, 830. Springer-Verlag, Berlin, 2007.

\bibitem[Hup]{Hupp} {\sc B.~Huppert:} Endliche Gruppen I. Die Grundlehren
der Mathematischen Wissenschaften,
 Band 134. Springer-Verlag, Berlin-New York, 1967.

\bibitem[Ka]{Ka} {\sc J. Kastl:} Inverse categories. in: Algebraische
Modelle, Kategorien und Gruppoide (ed. H.-J. Hoehnke),
 Akademie-Verlag, Berlin (1979), 51--60.

\bibitem[L]{L} {\sc M.~Linckelmann:} On inverse categories and transfer in
cohomology. to appear in
 Skye conference proceedings.

\bibitem[LS]{LS} {\sc M.~Linckelmann, M.~Stolorz:} 
   On simple modules over twisted finite category algebras. 
   Proc.\ Amer.\ Math.\ Soc.\~140 (2012), 3725--3737.


\bibitem[MP]{MartinoPriddy} {\sc J.~Martino, S.~Priddy:} 
    Stable homotopy classification of $BG^{\wedge}_p$. 
    Topology~{\bf 34} (1995), 633--649.

\bibitem[RS]{RS} {\sc K.~Ragnarsson, R.~Stancu:} Saturated fusion systems as idempotents in the double Burnside ring. Preprint. arXiv:0911.0085.


\bibitem[St]{St} {\sc B.~Steinberg:} M\"obius functions and semigroup
representation
theory. II. Character formulas and multiplicities. Adv. Math. {\bf 217}
(2008), 1521--1557.

\bibitem[T]{Thevenaz} {\sc J.~Th\'evenaz:} Endo-permutation modules, a guided tour. In: Group representation theory, 115--147, EPFL Press, Lausanne, 2007.

\bibitem[W]{W} {\sc P.~Webb:} Stratifications and Mackey functors. II:
globally defined Mackey functors.
  J. K-Theory {\bf 6} (2010), 99--170.

\bibitem[Y]{Yoshida} {\sc T.~Yoshida:} 
    On the unit groups of Burnside rings.  
    J. Math. Soc. Japan {\bf 42} (1990), 31--64.


\end{thebibliography}
\end{document}